\documentclass[11pt,centertags,oneside]{amsart}
\usepackage{amsmath,amstext,amsthm,amscd,typearea,hyperref}
\usepackage{amssymb}
\usepackage{a4wide}
\usepackage[mathscr]{eucal}
\usepackage{mathrsfs}
\usepackage{typearea}
\usepackage{charter}
\usepackage{pdfsync}

 \usepackage[normalem]{ulem} 

\usepackage[a4paper,width=16.5cm,top=3cm,bottom=3cm]{geometry}

\numberwithin{equation}{section}

\allowdisplaybreaks
\emergencystretch=\maxdimen
\usepackage{multicol}
\usepackage{xcolor}

\newtheorem{theorem}{Theorem}[section]

\newtheorem{proposition}[theorem]{Proposition}
\newtheorem{corollary}[theorem]{Corollary}
\newtheorem{lemma}[theorem]{Lemma}
\newtheorem{remark}[theorem]{Remark}

\newtheorem{question}[theorem]{Question}

\newcommand{\ZZ}{{\mathbb{Z}}}
\newcommand{\GG}{{\mathbf{G}}}

\newcommand{\RR}{{\mathbb{R}}}

\newcommand{\QQ}{{\mathbb{Q}}}

\DeclareMathOperator{\vol}{vol}

\newcommand{\GL}{\mathrm{GL}}
\newcommand{\SL}{\mathrm{SL}}

\newcommand{\SO}{{\mathrm{SO}}}

\newcommand{\Id}{\mathrm{Id}}

\newcommand{\nor}{\mathrm{N}}

\newcommand{\Ad}{{\mathrm{Ad}}}
\newcommand{\ad}{{\mathrm{ad}}}
\newcommand{\Lie}[1]{\mathfrak{\lowercase{#1}}}

\newcommand{\ov}{v}

\newcommand{\vep}{\varepsilon}

\providecommand{\vol}{\mathrm{vol}}

\newcommand{\diag}{\mathrm{diag}}

\newcommand{\tr}{\mathrm{tr}}

\newcommand{\lt}{\ltimes^{}_{l}}

\newcommand{\q}{{\mathbf q}}

\newcommand{\lin}{{\mathbf l}}

\newcommand{\X}{\mathscr{X}}

\renewcommand{\O}{\mathrm{O}}

\newcommand{\N}{\mathcal{N}}

\providecommand{\RR}{\mathbb{R}} \providecommand{\ZZ}{\mathbb{Z}}

\definecolor{han}{rgb}{1.0, 0, 0}

\newcommand{\pairs}{\mathscr{S}}
\newcommand{\gl}{\mathfrak{gl}}

\newcommand{\op}{\mathrm{op}}

\newcommand{\origin}{0}

\newcommand{\C}{\mathcal{C}}

\newcommand{\rem}[1]{ }

\usepackage{fancyhdr}

\title{Asymptotic distribution for pairs of linear and quadratic forms at integral vectors}

\author{Jiyoung Han}
\address{School of Mathematics, Korea Institute for Advanced Study}
\email{jiyounghan@kias.re.kr, hanjiwind@gmail.com}

\author{Seonhee Lim}
\address{Department of Mathematical Sciences and Research Institute of Mathematics, Seoul National University}
\email{slim@snu.ac.kr,seonhee.lim@gmail.com}

\author{Keivan Mallahi-Karai}
\address{School of Science, Constructor University, Campus Ring I, 28759 Bremen, Germany}
\email{kmallahikarai@constructor.university}

\begin{document}

\begin{abstract}
We study the joint distribution of values of a pair consisting of a quadratic form $\q$ and a linear form $\lin$ over the set of integral vectors, a problem initiated by Dani and Margulis \cite{DM2}. In the spirit of the celebrated theorem of Eskin, Margulis, and Mozes on the quantitative version of the Oppenheim conjecture, we show that if $n \ge 5$ then under the assumptions that for every $(\alpha, \beta ) \in \RR^2  \setminus \{ (0,0) \}$, the form $\alpha \q + \beta \lin^2$  is irrational and that
the signature of the restriction of $\q$  to the kernel of  $\lin$ is $(p, n-1-p)$, where $3\le p\le n-2$, the number of vectors $v \in \ZZ^n$ for which $\|v\| < T$, $a < \q(v) < b$ and $c< \lin(v) < d$ is asymptotically
 $$
 C(\q, \lin)(d-c)(b-a)T^{n-3}
 $$
as $T \to \infty$, where $C(\q, \lin)$ only depends on $\q$ and $\lin$. The density of the set of joint values
of $(\q, \lin)$ under the same assumptions is shown by Gorodnik \cite{Go}. 
\end{abstract}

\clearpage\maketitle
\thispagestyle{empty}


\medskip

\noindent\textbf{Mathematics Subject Classification 2010:} \texttt{60B15}.

\setcounter{tocdepth}{1}

\tableofcontents

\section{Introduction}
 
The Oppenheim conjecture, settled by Gregory Margulis in 1986 \cite{Mar}, states that for any non-degenerate irrational indefinite quadratic form $\q$ over $\RR^n$,  $n \ge 3$, the set $\q(\ZZ^n)$ of values of $\q$ over integral vectors is a dense subset of $\RR$. 

Margulis' proof uses the dynamics of Lie group actions on homogeneous spaces. More precisely, he shows that every pre-compact 
orbit of the orthogonal group $\SO(2,1)$ on the homogeneous space $\SL_3(\RR)/\SL_3(\ZZ)$ is compact. 
This proof also settled a special case of Raghunathan's conjecture on the action of unipotent groups on homogenous spaces. Raghunathan's conjecture was posed in
the late seventies (appearing in print in \cite{Dani81}) suggesting a different route towards resolving the Oppenheim conjecture. This conjecture was later settled in its full generality by Marina Ratner. 

Ever since Margulis' proof, homogenous dynamics has turned into a powerful machinery for studying similar questions
of number theoretic nature. In particular, various extensions and refinements of the Oppenheim conjectures have been studied. In the quantitative direction, one can inquire about the distribution of values of $\q(\ZZ^n \cap B(T))$, where $B(T)$ denotes the ball of radius $T$ centered at zero.  It was shown in a groundbreaking work \cite{EMM} by Eskin, Margulis, and Mozes that the number $\N_{T, I}(\q)$ of vectors $v \in B(T)$ with $\q(v) \in I:= (a,b)$
satisfies the asymptotic formula 
\begin{equation}\label{qOpp}
\N_{T, I}(\q) \sim C(\q) (b-a) T^{n-2}\;\text{as}\;T\rightarrow \infty,
\end{equation}
assuming that $\q$ is non-degenerate, indefinite and irrational, and has signature different from $(2,1)$ and $(2,2)$. 
Prior to \cite{EMM}, an asymptotically exact lower bound was established by Dani and Margulis \cite{DM} under the condition $n \ge 3$.

It is noteworthy that \eqref{qOpp} does not hold for all irrational quadratic forms of signatures
$(2,1)$ and $(2,2)$. However, for quadratic forms of signature $(2,2)$ that are not well approximable by rational forms, an analogous quantitative result for a modified counting function has been established in \cite{EMM2}. The question for forms of signature $(2, 1)$ remains open. 

Let $\q$ be an indefinite quadratic form of signature $(p,q)$.
The approach taken up in \cite{EMM}{ translates the problem of determining 
the asymptotic distribution of $\q(\ZZ^n)$} to the question of studying the distribution of translated orbits $a_tKx_0$
in the space $\SL_n(\RR)/\SL_n(\ZZ)$ of unimodular lattices in $\RR^n$.
Here, $a_t$ is a one-parameter diagonal subgroup of the orthogonal group $\SO(p, q)$
defined in \eqref{$a_t$}, $K$ is isomorphic to the maximal compact subgroup of the connected component of identity in $\SO(p, q)$, and $x_0 \in \SL_n(\RR)/\SL_n(\ZZ)$ is determined by the quadratic form $\q$.  One of the major challenges of the proof is that the required equidistribution result involves integrals of {\it unbounded} observables (or test functions). This difficulty is overcome by introducing a set of height functions which can be used to track the elements $k \in K$ for which the lattice $a_tkx_0$ has a large height, and thereby reducing the problem to bounded observables.
\rem{should we say something about Sargent? }

\subsection{Pairs of quadratic and linear forms}
In this paper  we study the joint distribution of the values of pairs $(\q, \lin)$ consisting of a quadratic and a linear form. This problem was first studied by Dani and Margulis \cite{DM2} who proved a result for the density of the joint values of pairs of a quadratic form
and a linear form in three variables. This result was extended by Gorodnik \cite{Go} to  forms with $n \ge 4$ variables. Our goal in this paper is to prove a quantitative version of these qualitative results. 

Fix $n \ge 4$, and write $\q$ for a non-degenerate indefinite quadratic form on $\RR^n$ and  $\lin $ for a nonzero linear form on $\RR^n$. 
Denote by $\pairs^0_n$ the set of all such pairs $(\q, \lin)$ satisfying the following two conditions:
\begin{enumerate}\label{eqn:1}
\item[(A)] The restriction of $\q$ to the subspace defined by ${\lin=0}$ is indefinite.
\item[(B)] For every $(\alpha, \beta ) \in \RR^2  \setminus \{ (0,0) \}$, the form $\alpha \q + \beta \lin^2$  is irrational.
\end{enumerate}
The main result (Theorem 1) of \cite{Go} shows that under these assumptions, the set of joint values $$\{ (\q(v), \lin(v)): v \in \ZZ^n \} \subseteq \RR^2$$ 
is dense. Note that condition (A) is necessary for the set of values to be dense in $\RR^2$. Condition (B), however, can conceivably be weakened, see a remark in Section 6 of \cite{Go}.

Our goal in this work is to study a quantitative refinement of this problem. More precisely, we will ask the following question:

\begin{question}\label{q-counting}
For  $(\q, \lin) \in \pairs^0_n$ and intervals $I= (a, b), J=(c,d)$, denote by $\N_{T,I, J}(\q, \lin)$ the number of vectors $v \in \ZZ^n$ for which $\|v\| < T$, $ \q(v) \in I $ and $ \lin(v) \in J$. Find conditions under which the following asymptotic behavior holds:
\[ \N_{T,I, J}(\q, \lin) \sim  C(\q, \lin) \ (b-a)(d-c) T^{n-3} . \]
as $T \to \infty$. Here,  $C(\q, \lin)$ is a positive constant that depends only on $\q$ and  $\lin$.
\end{question}

Note that the above asymptotic is consistent with the general philosophy in \cite{EMM}. 
The ball $B(T)$ of radius $T$ centered at zero contains about $T^n$ integral vectors. 
As $v$ ranges in $B(T)$, $\q(v)$ takes values
in an interval of length about $T^2$, while the values of $ \lin(v)$ range in an interval of length comparable to $T$. Packing the $T^n$ 
points $(\q(v), \lin(v))$ in a box of volume comparable to $T^3$, one might expect that a rectangle of fixed size is hit about $T^{n-3}$ times.

\subsection{Statement of results}
Let $| I |$ denote the length of the interval $I \subseteq \RR$. 
Our main result is the following. 

\begin{theorem}\label{thm:main}
Let $\q$ be a non-degenerate indefinite quadratic form on $\RR^n$ for $n \ge 5$  and let $\lin$ be a nonzero linear form on $\RR^n$.   For $T>0$, open bounded intervals $I, J \subseteq \RR$, let $\N_{T, I, J}(\q, \lin)$ denote the number of vectors $v$ for which $$\|v\| < T, \quad  \q(v)  \in I , \quad   \lin(v)  \in J.$$ 

Let $\pairs_n$ be the set of $(\q, \lin)\in \pairs^0_n$ for which the restriction of $\q$ to $\ker \lin$ is non-degenerate in $\ker\lin$ and not of signature $(2,2)$. Then for any $(\q, \lin)\in \pairs_n$, we have
$$ \lim_{T \to \infty}  \frac{\N_{T, I, J}(\q, \lin)}{T^{n-3}}   = C(\q, \lin) |I| \ |J|,$$
where $C(\q, \lin)$ is a positive constant depending only on $\q$ and $\lin$.
\end{theorem}

\begin{remark} 
Theorem \ref{thm:main} does not generally hold if the restriction of $\q$ to $\ker \lin$ is of signature $(2,2)$, see Section \ref{sec:counter} for a counterexample. 
Based on the main result of \cite{EMM2} it seems reasonable that a modified result under certain diophantine condition, might still hold.  
\end{remark}


\subsection{Strategy of proof}
The proof follows the same roadmap as in \cite{EMM}. We will start by translating the question into one about the distribution of translated orbits on homogenous spaces. 

It is well known that the space $\X_n$ of unimodular lattices in $\RR^n$ can be identified with the homogenous space $\SL_n(\RR)/\SL_n(\ZZ)$.  This space is non-compact and carries an $\SL_n(\RR)$-invariant probability measure. In many problems in homogenous dynamics, it is useful to quantify the extent to which a lattice lies in the cusp of 
 $\X_n$. 

We will translate Question \ref{q-counting} to the problem of showing that certain translated orbits of the form $a_tK \Lambda$ become asymptotically equidistributed in $\X_n$ as $t$ goes to $\infty$.
Here, $K$ is the maximal compact subgroup of the connected component of identity in $\SO(p,q-1)$, where $(p,q-1)$ denotes the signature of the restriction of $\q$ to $\ker \lin$. At this point, several problems will arise. On the one hand, the existence of various intermediate subgroups
make the application of Dani-Margulis theorem more difficult. Dealing with this problem
requires us to classify all intermediate subgroups that can arise. The second problem, similar to the one in \cite{EMM}, involves the unboundedness of test functions to which the equidistribution result must be applied. We will adapt the technique used in \cite{EMM} with one twist. Namely, we will prove a boundedness theorem for the integrals of 
$ \alpha( a_tk \Lambda)^s$ for some $s>1$, where $ \alpha$ is the Margulis height function defined as follows:
for a lattice $\Lambda$, 

$$ \alpha ( \Lambda)= \max \left\{ { \| v \|}^{-1}: v \in \Omega( \Lambda) \right\},$$
where $$\Omega(\Lambda)=\left\{v=v_1\wedge \cdots \wedge v_i : v_1, \ldots, v_i \in \Lambda, \quad 1 \le i \le n \right\} \setminus \{ 0 \}.$$ 

More precisely, we will show that for $p\ge 3$, $q\ge 2$ and $0<s<2$, for every $g \in \SL_n(\RR)$ we have 
\[
\sup_{t>0} \int_K \alpha(a_tk.g\ZZ^n)^s dm(k) < \infty.
\]

The strategy in \cite{EMM} requires $K$ not to have non-trivial fixed vectors in certain representation spaces. 
Since this is no longer the case here, we need to use a refined version of the $\alpha$ function developed by Benoist-Quint  \cite{BQ}, \cite{Sar}  which we recall now.

Let $H$ be a connected semi-simple Lie subgroup of $\SL_n(\RR)$. Denote by $\bigwedge(\RR^n)$ the exterior power of $\RR^n$, that is, the direct sum of all $ \bigwedge^i(\RR^n)$ for $ 0 \le i \le n$. Let $\rho : H \rightarrow \GL(\bigwedge\RR^n)$ be the representation of $H$ induced by the linear representation of $H$ on $\RR^n$.
 
Since $H$ is semi-simple, $\rho$ decomposes into a direct sum of irreducible representations of $H$ parametrized by their highest weights $\lambda$.  For each $\lambda$, denote by $V^{\lambda}$ the direct sum of all irreducible sub-representations of $\rho$ with highest weight $\lambda$. Denote by $\tau_{\lambda}$ the 
canonical orthogonal projection of $\bigwedge(\RR^n)$ onto $V^{\lambda}$. 
Fix $\vep>0$. Following  \cite{BQ}, \cite{Sar} define  {\it Benoist-Quint $\varphi$-function} 
$$\varphi_{\vep} : \bigwedge(\RR^n)\mapsto [0, \infty]$$
for $v \in \bigwedge^i(\RR^n), 0<i <n$, by
$$\varphi_{\vep} (v) = \left\{\begin{array}{ll}
    \min_{\lambda\neq 0} \vep^{(n-i)i}\|\tau_{\lambda}(v)\|^{-1}, & \hbox{if $\|\tau_0(v)\|\leq \vep^{(n-i)i}$;} \\
    0, & \hbox{otherwise.}
  \end{array}\right.
$$

Let us define $f_{\vep}: \SL_n(\RR)/\SL_n(\ZZ) \to [0, \infty]$ by
$$f_\vep(\Lambda)=  \max \left\{  \varphi_{\vep} (v): v \in \Omega(\Lambda) \right\}. $$


\subsection{Outline of the paper}
This paper is organized as follows. In Section \ref{sec:intermediate}, after recalling some preliminaries we state and prove
results about the equidistribution of translated orbits of the form $a_tKg\ZZ^n$ in the orbit closure.\rem{ which orbit closure?}This requires us to classify all the intermediate subgroups that can  potentially appear in the conclusion of Ratner's theorem. 
In Section \ref{section:alpha}, we recall Siegel's integral formula and prove Theorem \ref{Siegel integral formula:proper F}, which is an analog for a subset of lattices that all share a rational vector. 
This proof relies on the boundedness of some integrals (see Theorem \ref{upperbound of alpha}) involving $\alpha$-function, which is proven in the beginning of this section. In Section \ref{sec:bounds} we will show that 
the integral of the $\alpha$-function along certain orbit translates is uniformly bounded. This is one of the major ingredients of the proof. In Section \ref{sec:bridge}, we will use results of the previous sections to establish Theorem \ref{thm:main}.
Finally, Section \ref{sec:counter} is devoted to presenting counterexamples illustrating that the analog of Theorem \ref{thm:main} does not hold for
certain forms of signatures $(2, 2)$ and $(2,3).$

\vspace{2mm}

{\it Acknowledgement.} 
We would like to thank the referee for carefully reading the manuscript and for providing suggestions that improved the content and the exposition of the paper.
This work was partly done during authors' visits to Seoul National University and Jacobs University Bremen. 
The first author is supported by a KIAS Individual Grant 
MG088401 at Korea Institute for Advanced Study.
The second author is supported by National Research Foundation of Korea under Project number NFR-2020R1A2C1A01011543, SSTF-BA1601-03, and Korea Institute for Advanced Study.

 \section{Equidistribution results}\label{sec:intermediate}
{ In this section, we will relate Question \ref{q-counting} to the question of equidistribution of certain orbit translates in homogeneous spaces. 
In Subsection \ref{classification}, we recall some preliminaries
and in Subsection \ref{notation-groups}, we establish a connection to the homogeneous dynamics. }

\subsection{Preliminaries: canonical forms for pairs $(\q, \lin)$ and their stabilizers}\label{classification}
In this subsection we will first introduce some notation and recall a number of basic facts about the space of unimodular lattices in $\RR^n$. Then we will recall the classification in \cite{Go} of pairs consisting of a quadratic form and a linear form under the action of $\SL_n(\RR)$.

Let $\q$ be a non-degenerate isotropic quadratic form on $\RR^n$. There  exists $ 1 \le p \le n-1$,  $ \lambda \in \RR \setminus \{ 0 \}$, and $g \in \SL_n(\RR)$ such that 
$$ \lambda \cdot \q(gx)= 2x_1x_2+ x_3^2+ \cdots + x_{p+1}^2 -(x_{p+2}^2 + \cdots + x_n^2).$$
We say that $\q$ has signature $(p, n-p)$. 
We need a similar classification for pairs of quadratic and linear forms. 
Let $\q$ be as above, and let $\lin$ be a nonzero linear form on $\RR^n$. For $g \in \SL_n(\RR)$ define 
$$\q^g(x)= \q(gx), \qquad \lin^g(x)= \lin(gx).$$
For $i=1,2$, let $\q_i$ and $\lin_i$ be as above. We say that $(\q_1, \lin_1)$ is equivalent to  $(\q_2, \lin_2)$  if $\q_1= \lambda \cdot \q_2^g$ and $\lin_1 = \mu \cdot \lin_2^g$ for some $g \in \SL_n(\RR)$ and nonzero scalars $\lambda$ and $\mu$. 
Using the action of $\SL_n(\RR)$ we can transform any pair $(\q, \lin)$ into a standard pair.

\begin{proposition}[{\cite[Proposition 2]{Go}}]\label{def:types}
Every pair $(\q, \lin)$ as above is equivalent to one and only one of the following:
\begin{align}\label{eqn:type}
\mathrm{(I)} \;\;\;  & (2 x_1 x_2 + x_3^2+ \cdots + x_{p+1}^2 - x^2_{p+2}+  \cdots - x_{n}^2, x_n) \; \quad  p=1, \cdots, n-1, \\
\mathrm{(II)} \;\;\; & (2x_1 x_2 + x_3^2+ \cdots + x_{p}^2 - x_{p+1}^2-\cdots -x_{n-2}^2+2x_{n-1}x_n, x_n) \;  \quad p=1, \cdots, [n/2-1] .
\end{align}
\end{proposition}

Pairs in (2.1) and (2.2) are referred to as type I and II, respectively. It can be seen that the  
pair $(\q, \lin)$ is of type I if and only if the restriction of $\q$ to $\ker \lin$ is non-degenerate. 
In this paper we deal only with pairs $(\q, \lin)$ of type I satisfying (A) and (B). We denote this set by $\pairs_n$. 

\begin{remark}
For pairs of type II satisfying (A) and (B), it appears that the maximal compact subgroup $K$ preserving both $\q$ and $\lin$ is not sufficiently large for our methods to apply. 
\end{remark}

\subsection{Connection to the homogenous dynamics and equidistribution results} \label{notation-groups}
Let $G=\SL_n(\RR)$ and $\Gamma= \SL_n(\ZZ)$.
Denote the Lie algebra of $G$ by $\mathfrak{sl}_n(\RR)$. 
Suppose that $(\q, \lin)$ is equivalent to 
\begin{equation}\label{canonic1}
(\q_0, \lin_0) = (2 x_1 x_2 + x_3^2+ \cdots + x_{p+1}^2 - x_{p+2}^2 - \cdots - x_{n}^2, x_n).
\end{equation}

Let $H$ be the subgroup of $\SL_n(\RR)$ defined by
\[
H=
\left(\begin{array}{c|c}
 & 0 \\
 \SO(p,q-1)^\circ  & \vdots \\
 & 0 \\
\hline
0 \quad \cdots \quad 0 & 1
\end{array}\right).
\]
Denote by $\SO(\q_0,\lin_0)$ the subgroup of $\SO(\q_0)$ that stabilizes $\lin_0$, where $(\q_0, \lin_0)$ is defined as in \eqref{canonic1} so that
$\SO(\q_0,\lin_0)^\circ$ is isomorphic to $H$. 
The Lie algebra of $H$, denoted by $\Lie{h}$, consists of the sub-algebra consisting of matrices of the form

\begin{equation}\label{def:H}
\Lie{h}=
\left(\begin{array}{c|c}
 & 0 \\
 \Lie{so}(p,q-1)  & \vdots \\
 & 0 \\
\hline
0 \quad \cdots \quad 0 & 0
\end{array}\right).
\end{equation}

It is not difficult to see that $K := H \cap \SO(n)$ is a maximal compact subgroup of $H$ and is isomorphic to $\SO(p)\times \SO(q-1)$. Denote the canonical basis of $\mathbb{R}^n$ by $\{e_1, \dots, e_n \}$. Let $a_t$ denote the one-parameter subgroup defined by 
\begin{equation}\label{$a_t$} a_t e_1= e^{-t} e_1, \quad a_t e_2= e^t e_2, \quad a_t e_j = e_j, \quad 3 \le j \le n. \end{equation}

Using this notation we can state one of the main results of this paper. 

\begin{theorem}\label{alphabounded-intro}\label{main1}
For $p\ge 3$, $q\ge 2$ and $0<s<2$. Then for every $ \Lambda \in \X_n$ we have 
\[
\sup_{t>0} \int_K \alpha(a_tk \Lambda)^s dm(k) < \infty.
\]
 
\end{theorem}

This theorem is analogous to Theorem 3.2 in \cite{EMM}. What makes the proof of Theorem \ref{main1} more difficult is that the integration is over a {\it proper} subgroup of $\SO(p) \times \SO(q)$. In general, one can see that if $K$ is replaced by an arbitrary subgroup of $\SO(p) \times \SO(q)$ with  large  co-dimension, then the analog of Theorem \ref{main1} may not hold. As a result, establishing the boundedness of the integral requires a more delicate analysis of the excursion to the cusp of the translated orbit $a_tK \Lambda$.  Using Theorem \ref{main1}, we will prove the theorem below from which  Theorem \ref{thm:main} will be deduced. 

\begin{theorem}\label{main2}
Suppose $p\ge 3$, $q\ge 2$ and $s>1$. Let $\phi: \X_n \to \RR$ be a continuous function such that 
\[ | \phi( \Lambda) | \le C \alpha( \Lambda)^s \]
for all $ \Lambda \in \X_n$ and some constant $C>0$. Let $ \Lambda \in \X_n$ be such that $\overline{H \Lambda }$ is either $\X_n$ or is  of the form $ (\SL_{n-1}(\RR) \ltimes_l \RR^{n-1}) \Lambda$, where $\SL_{n-1}(\RR) \ltimes_l \RR^{n-1}$ is defined by \eqref{semidirect}. 
Then, 
\[ \lim_{ t \to \infty} \int_K \phi(a_tk \Lambda) \ dm(k) = \int_{_{\overline{H \Lambda }} } \phi \ d\mu_{ _{\overline{H \Lambda }}}, \]
where $\mu_{ _{\overline{H \Lambda }}}$ is the $H$-invariant probability measure on $\overline{H \Lambda }$. 
\end{theorem}

%
%

%
%
%

We shall see that Theorem \ref{main2} will apply to $ \Lambda=g_0 \ZZ^n$, when 
$(\q_0^{g_0}, \lin_0^{g_0}) \in \pairs_n$, see Theorem \ref{intermediate subgroup}.

The methods used are inspired by the ones employed in \cite{EMM}.  We will recall a theorem of Dani and Margulis after introducing some terminology and set some
notations. Let $G$ be a real Lie group with the Lie algebra $\Lie{g}$. Let  $\Ad: G \to \GL(\Lie{g} )$ denote the adjoint representation of $G$.
An element $g \in G$ is called Ad-unipotent if $\Ad(g)$ is a unipotent linear transformation.  A one-parameter group $\{ u_t \}$ is called
\emph{Ad-unipotent} if every $u_t$ is an Ad-unipotent element of $G$. In this section we will recall some results from \cite{DM} and \cite{EMM} that will be needed in the sequel.


As in the proof of the quantitative Oppenheim conjecture \cite{EMM}, a key role is played by Ratner's equidistribution theorem. Suppose $G$ is a connected Lie group, $\Gamma<G$ a lattice, and $H$ is a connected subgroup of $G$ generated by unipotent elements in $H$. 
Ratner's orbit closure theorem asserts that  for every point $x \in G/\Gamma$ there exists a connected closed subgroup $L$ containing $H$ such that  $\overline{Hx}= Lx$. Moreover, $Lx$ carries an $L$-invariant probability measure $\mu_{L}$.
In order to apply Ratner's theorem in concrete situations, one needs to be able to classify all subgroups $L$ that can arise. In the next subsection, we will classify all connected subgroups of $\SL_n(\RR)$ containing $H$. Using well known results in Lie theory, this classification problem is equivalent to the problem of classifying all Lie sub-algebras of $\Lie{sl}_n(\RR)$ containing $\Lie{h}$.

\subsection{Intermediate subgroups}
We will maintain the notation as in Section \ref{notation-groups}.
Since $\Lie{h}$ is semisimple, $\Lie{sl}_n(\RR)$, regarded as an $\ad(\Lie{h})$-module, can be decomposed as the direct sum of irreducible $\ad(\Lie{h})$-invariant subspaces. 
 For $ 1 \le i, j \le n$, let $E_{ij}$ be the $n \times n$ matrix whose only nonzero entry is $1$ and is located on the $i$th row and $j$th column.  We will refer to
 $\{ E_{ij}: 1 \le i, j \le n \}$ as the canonical basis of the Lie algebra $\gl_n(\RR)$.

\begin{proposition}\label{prop:decomp}
The Lie algebra $\Lie{sl}_n(\RR)$ splits as the direct sum of irreducible $\ad(\Lie{h})$-invariant subspaces \[
\Lie{sl}_n(\RR)
=\Lie{h} \oplus \Lie{s} 
\oplus \Lie{u}^+ \oplus \Lie{u}^-
\oplus \Lie{t},
\] 
where 
\begin{itemize} 
\item $\Lie{s}$ consists of all matrices of the form 
\[ \left(\begin{array}{cc|c}
A & B & 0 \\
-B^{t} & D & 0 \\
\hline
0 & 0 & 0
\end{array}\right) \]
and 
$A$ and $D$ are symmetric matrices of size $p$ and $(q-1)$, respectively, such that $\tr(A)+\tr(D)=0$, and  $B$ is an arbitrary 
$p$ by $q-1$ matrix. 
\end{itemize}

\begin{itemize}
\item $\Lie{u}^+$ is the $(n-1)$-dimensional subspace  spanned by $ E_{in}, 1 \le i \le n-1$. 
\item $\Lie{u}^-$ is the $(n-1)$-dimensional subspace  spanned by $ E_{ni}, 1 \le i \le n-1$. 

\item $\Lie{t}$ is the one-dimensional subspace spanned by $E_{11}+ \cdots + E_{n-1,n-1}-(n-1)E_{nn}$. 
 \end{itemize}
\end{proposition}
\begin{proof}
The only challenging assertion lies in demonstrating that an $\ad(\Lie{H})$-invariant subspace $\Lie{S}$ is $\ad(\Lie{H})$-irreducible.
Using the weight decomposition of $\Lie{sl}_n(\RR)$ for the restricted root system of $\Lie{h}$, one can establish this assertion by showing that any weight vector of $\Lie{s}$ can be transformed into another weight vector via the adjoint action of restricted roots (for further elaboration, refer to \cite{Han24}).
\end{proof}

 Let $\Phi:\Lie{u}^+ \rightarrow \Lie{u}^-$ map $E_{in}$ to $E_{ni}$ for
 $ 1 \le i \le p$ and $E_{in}$ to $- E_{ni}$ for $ p+1 \le i \le n-1$. In other words, %
\[ \Phi  \left( \sum_{ i =1}^{p}  v_i E_{in}  + \sum_{ i =p+1}^{n-1} v_i  E_{in} \right) :=  \sum_{ i =1}^{p}  v_i E_{ni}  - \sum_{ i =p+1}^{n-1} v_i  E_{ni}. \]

One can verify that $\Phi$ is an $\Lie{h}$-module isomorphism. For any nonzero $\xi \in \RR$, consider the subspace  
\[
\Lie{u}^{\xi}: = (\Id + \xi \Phi) \Lie{u}^+. \]
It is clear that $\Lie{u}^{0}= \Lie{u}^+$. Set also $\Lie{u}^{\infty}:= \Lie{u}^-$. Note that 
for $ \xi \neq 0, \infty$, the subspace $\Lie{u}^{\xi}$ is not a \emph{subalgebra} of $\Lie{sl}_n(\RR)$.
\begin{remark}\label{intermediate quadratic form}
Define the quadratic form $\q_{\xi}$ by
\[
\q_{\xi}(v)=(x_1^2+\cdots+x_p^2-x_{p+1}^2-\cdots-x_{n-1}^2)+\xi x_n^2.
\] 
The Lie algebra $\Lie{so}( \q_{\xi})$ for $ \xi\in\RR \setminus \{ 0 \}$ decomposes as 
$\Lie{so}( \q_{\xi})= \Lie{H}\oplus \Lie{u}^{\xi}$.
Moreover, any quadratic form $\q'$ for which $\SO(\q')$ contains $H$ is of the form $\q_\xi$ up to scalar multiplication.
\end{remark}


\begin{proposition}\label{intermediate type I}
Let $\Lie{f}$ be a subalgebra of $\Lie{sl}_n(\RR)$ containing $\Lie{h}$.
Then $\Lie{f}$ is one of the Lie algebras in Table 1. 
\begin{table}[h!]
\begin{center}
\begin{tabular}{c|l}
\hline
\rule{0in}{0.2in}
Levi subalgebra & $\quad\Lie{f}$\\
\hline
\rule{0in}{0.2in}
$\Lie{h}\simeq \Lie{so}(p,q-1)$ & 
$\Lie{h}$,\; 
$\Lie{h} \oplus \Lie{t}$,\; 
$\Lie{h} \oplus \Lie{u}^+$,\; $\Lie{h} \oplus \Lie{u}^-$,
$\Lie{h} \oplus \Lie{u}^+ \oplus \Lie{t}$, $\Lie{h} \oplus \Lie{u}^- \oplus \Lie{t}$
\;
 \\
\hline
\rule{0in}{0.2in}
$ \Lie{so}(\q_\xi)$ &
$ \Lie{so}(\q_\xi)$  $  \quad \xi \in \RR \setminus \{ 0  \}$\\
\hline
\rule{0in}{0.2in}
$\Lie{h} \oplus\Lie{s}\simeq \Lie{sl}_{n-1}(\RR)$ & 
$\Lie{sl}_{n-1}(\RR)$,\;
$\Lie{sl}_{n-1}(\RR)\oplus\Lie{u}^+$,\;
$\Lie{sl}_{n-1}(\RR)\oplus\Lie{u}^-$,\;
$\Lie{sl}_{n-1}(\RR)\oplus\Lie{t}$,\;\\
&
$\Lie{sl}_{n-1}(\RR)\oplus\Lie{u}^+\oplus \Lie{t}$,\;
$\Lie{sl}_{n-1}(\RR)\oplus\Lie{u}^-\oplus \Lie{t}$\\
\hline
\rule{0in}{0.2in}
$\Lie{sl}_n(\RR)$ & 
$\Lie{sl}_n(\RR)$ \\
\hline
\end{tabular}
\end{center}
\vspace{0.05in}
\caption {List of intermediate sub-algebras} \label{Table type I}
\end{table}
\end{proposition}

\vspace{-0.4in}
\begin{proof}
Before we start the proof, let us recall that 
$$ \Lie{H}\oplus \Lie{u}^{\xi}= \Lie{so}( \q_{\xi}), \quad  \xi \in \RR, \qquad 
\Lie{h} \oplus\Lie{s}\simeq \Lie{sl}_{n-1}(\RR).$$
Let $ \Lie{f}$ be as in the statement of Proposition \ref{intermediate type I}. Since $\Lie{h}$ is semisimple and $\Lie{f}$ is an $\Lie{h}$-submodule of $\Lie{sl}_n(\RR)$, $\Lie{f}$ decomposes into a direct sum of $\Lie{h}$ and irreducible $\Lie{h}$-invariant subspaces, each isomorphic to one of $\Lie{s}, \Lie{u}^+, \Lie{u}^-$ and $\Lie{t}$. Note that aside from $\Lie{u}^+$ and $ \Lie{u}^-$, which are isomorphic $ \Lie{h}$-modules, no other two of these $\Lie{h}$-modules are isomorphic. 
One can thus write $\Lie{f}= \Lie{f}_1 \oplus \Lie{f}_2$, where $\Lie{f}_1$ is a direct sum of
$\Lie{h}$ with a subset of $\{ \Lie{s}, \Lie {t} \}$, and $ \Lie{f}_2$ is an $\Lie{h}$-submodule of  $ \Lie{u}^+ \oplus \Lie{u}^-$. 
We will consider several cases. First assume that $ \Lie{f}_1=\Lie{h}$. All $\Lie{h}$-submodules of  $ \Lie{u}^+ \oplus \Lie{u}^-$
are of the form $\Lie{u}^{\xi}$ for $ \xi \in \RR \cup \{ \infty \}$. This leads to the submodules 
$\Lie{h} \oplus \Lie{u}^+$, $\Lie{h} \oplus \Lie{u}^-$ and  $ \Lie{H}\oplus \Lie{u}^{\xi}= \Lie{so}( \q_{\xi})$, all of which are subalgebras of $\Lie{sl}_n(\RR)$. 
Consider the case $\Lie{f}_1=\Lie{h} \oplus \Lie{t}$. One can easily see that $\Lie{h} \oplus \Lie{u}^+ \oplus \Lie{t}$, $\Lie{h} \oplus \Lie{u}^- \oplus \Lie{t}$
are both subalgebras of $\Lie{sl}_n(\RR)$. However, the inclusion $$[ \Lie{t}, \Lie{u}^{\xi} ]  \subseteq \Lie{u}^{-\xi}$$ rules out
the potential candidate $ \Lie{h} \oplus \Lie{u}^{\xi} \oplus \Lie{t}$. 
The case $\Lie{f}_1=  \Lie{h} + \Lie{s}= \Lie{sl}_{n-1}(\RR)$
can be dealt with similarly. In view of the inclusion  $$[  \Lie{ s}, \Lie{u}^{\xi} ]  \subseteq  \Lie{u}^{-\xi},$$ 
the potential candidates $ \Lie{sl}_{n-1}(\RR) \oplus \Lie{u}^{\xi}$ for $ \xi \neq 0, \infty$ are ruled out, while
$  \Lie{sl}_{n-1}(\RR),  \Lie{sl}_{n-1}(\RR) \oplus  \Lie{u}^{+}$ and $ \Lie{sl}_{n-1}(\RR) \oplus  \Lie{u}^{-}$ are all possible. The last 
case $ \Lie{f}_1=  \Lie{h} \oplus \Lie{s} \oplus \Lie{t}$ can be studied similarly.

%
%
%

\end{proof}

For a subgroup $F$ of $\SL_{n-1}(\RR)$, denote
 
\begin{equation}\label{semidirect}
F\ltimes_u \RR^{n-1}=\left(
\begin{array}{c|c}
 &  \\
F  & \RR^{n-1} \\
 &  \\
\hline
0 \:\cdots\: 0 & 1
\end{array}
\right)\quad\text{and}\quad 
F\ltimes_l \RR^{n-1}=\left(
\begin{array}{c|c}
 & 0 \\
F  & \vdots \\
 & 0 \\
\hline
\:\RR^{n-1}\: & 1
\end{array}
\right).
\end{equation}

\begin{theorem}[{Classification of possible orbit closures}]\label{intermediate subgroup}
Assume that $(\q, \lin)\in \pairs_n$. Let $g_0\in \SL_n(\RR)$ be such that $\SO(\q, \lin)^\circ=g_0^{-1}Hg_0$.
Let $F\le G$ denote the closed Lie subgroup containing $H$ with the property that $\overline{H g_0\Gamma}=g_0F\Gamma \subseteq G/\Gamma$. Then either $F=G$ or $F =g_0^{-1} (\SL_{n-1}(\RR) \ltimes_l \RR^{n-1}) g_0$.

\end{theorem}
One ingredient of the proof is the following theorem of Shah:

\begin{theorem}[{\cite[Proposition 3.2]{Shah}}]\label{Shah}
Let $\GG \le \SL_n$ be a $\QQ$-algebraic group and $G= \mathbf{G}(\RR)^{\circ}$. Set $\Gamma= \mathbf{G}(\ZZ)$, and let $L$ 
be a subgroup which is generated by algebraic unipotent one-parameter subgroups of $G$ contained in $L$. Let $ \overline{L\Gamma} 
= F \Gamma$ for a connected Lie subgroup $F$ of $G$. Let $  \mathbf{F}$ be the smallest algebraic $\QQ$-group containing
$L$. Then the radical of $ \mathbf{F}$ is a unipotent $\QQ$-group and $ F = \mathbf{F}(\RR)^{\circ}$. 
\end{theorem}

\begin{proof}[Proof of Theorem \ref{intermediate subgroup}]
The proof relies on Proposition~\ref{intermediate type I}. Recall that $\SO(\q,\lin)\simeq \SO(p,q-1)$ is semisimple and 
there are two proper $\SO(\q,\lin)$-invariant subspaces $\mathcal L_1$ and $\mathcal L_2$ in the dual space $(\RR^n)^*$ of $\RR^n$ with
$\dim \mathcal L_1=n-1$ and $\dim \mathcal L_2=1$.
Notice that since $(\q, \lin)\in \pairs_n$, $\mathcal L_2$ is an irrational subspace. 

Let $\Lie{F}=\mathrm{Lie}(F)$. After conjugation by $g_0$, since $\Lie{H}\subseteq \Lie{F}$, the Lie algebra $\Lie{F}$ is a subalgebra of $\Lie{sl}_n(\RR)$ appearing in Table \ref{Table type I} of Proposition ~\ref{intermediate type I}. We will show that if $\Lie{F}\neq \Lie{sl}_n(\RR)$, then only possible $\Lie{F}$ is $\Lie{SL}_{n-1}(\RR)\ltimes_{l} \RR^{n-1}$.

\vspace{2mm}

\noindent {\it Claim 1.} $\Lie{F}$ does not contain $\Lie{T}$.

By Theorem~\ref{Shah}, $F$ is (the connected component of) the smallest algebraic $\QQ$-group and the radical of $F$ is a unipotent algebraic $\QQ$-group. 
According to Table \ref{Table type I}, if $\Lie{T}\subseteq \Lie{F}$, the radical of $\Lie{F}$ is one of $\Lie{T}$, $\Lie{U}^+\oplus \Lie{T}$ or $\Lie{U}^-\oplus \Lie{T}$, which is not possible since $\Lie{T}$ is not unipotent.\\

\noindent {\it Claim 2.} $\Lie{U}^+$ is not contained in the radical of $\Lie{F}$.

If $\Lie{U}^+$ is in the radical of $\Lie{F}$, then $\Lie{F}$ is either $\Lie{H}\oplus\Lie{u}^+$ or $\Lie{sl}_{n-1}(\RR)\oplus\Lie{u}^+$.
In both cases, $F$ has invariant subspaces $\mathcal L_1$ and $\mathcal L_2$ in $(\RR^n)^*$.
Since $F$ is a $\QQ$-group, any $F$-invariant subspace in $(\RR^n)^*$ is defined over $\QQ$. In particular,  $\mathcal L_2$ must be a rational subspace, a contradiction.\\

\noindent {\it Claim 3.} $F$ is not semisimple.

If $F\lneq G$ is semisimple, then $F$ is either $\SO(\q,\lin)^\circ$ or $\SO(\q+\xi\lin^2)^\circ$ for some $\xi \in \RR-\{0\}$. If $F$ is $\SO(\q,\lin)^\circ$, $F$ has an invariant subspace $\mathcal L_2$ in $(\RR^n)^*$, which leads to a contradiction as in Claim 2. 
If $F\simeq \SO(\q+\xi\lin^2)^\circ$, since $F$ is defined over $\QQ$, $\q+\xi\lin^2$ is a scalar multiple of a rational form. This 
contradicts our assumption that $\alpha\q+\beta\lin^2$ is not rational for all nonzero $(\alpha,\beta)\in\RR^2$.\\

Thus, aside from  $\Lie{sl}_{n-1}(\RR)\oplus \Lie{U}^-$, the only possible option for $\Lie{F}$  is $\Lie{so}(p,q-1)\oplus \Lie{U}^-$.\\

\noindent {\it Claim 4.} Levi subgroup of $F$ is not  isomorphic to  $\SO(\q,\lin)^\circ$.

Suppose not. Let $L$ be a unipotent radical of $F$. 
By Levi-Malcev theorem (\cite{Mal}, see also \cite[Corollary 3.5.2]{AbMo}), there is $\ell \in L$ such that $\ell^{-1}\SO(\q,\lin)\ell=\SO(\q^{\ell}, \lin^{\ell})$ is a Levi subgroup of $F$, which is defined over $\QQ$.

Choose a basis $\lin_1, \ldots, \lin_{n-1}$ of $\mathcal L_1$ and $\lin_n$ of $\mathcal L_2$ such that $\q=\lin_1^2+\cdots+\lin_p^2-\lin_{p+1}^2-\cdots-\lin_{n}^2$.
Since the action of $L$ fixes elements of $\mathcal L_1$, the space $\mathcal L_1$ is an $\SO(\q^{\ell}, \lin^{\ell})$-invariant subspace which is defined over $\QQ$ by the assumption of $\SO(\q^{\ell}, \lin^{\ell})$. 
Choose a rationl linear form $\lin_0\in (\RR^n)^*$ such that $\langle\lin_0\rangle$ is $\SO(\q^{\ell}, \lin^{\ell})$-invariant and $(\RR^n)^*=\mathcal L_1 \oplus \langle\lin_0\rangle$. Clearly, $\lin_0=c\lin^{\ell}$ for some $c\in \RR-\{0\}$.
Moreover, by Remark~\ref{intermediate quadratic form}, since any quadratic forms fixed by $\SO(\q^{\ell}, \lin^{\ell})$ are of the form
\begin{equation*}
\q'=\alpha' (\lin_1^2+\cdots+\lin_p^2-\lin_{p+1}^2-\cdots-\lin_{n-1}^2)^{\ell} + \beta' \lin_0^2=\alpha' (\lin_1^2+\cdots+\lin_p^2-\lin_{p+1}^2-\cdots-\lin_{n-1}^2)+\beta'\lin_0^2,
\end{equation*}
there is a nontrivial $(\alpha',\beta')\in \RR^2$ such that $\q'$ is rational. Since $\mathcal L_1$ is an $(n-1)$-dimensional rational subspace of $(\RR^n)^*$, there is a rational vector $\ov\in \RR^n$ such that $\lin_j(\ov)=0$ for all $1\le j\le n-1$ and $\lin_0(\ov)\neq 0$.
Evaluating $\q'$ on $\ov$, we have $\beta'\lin_0(\ov)\in \QQ$ so that $\beta'$ is a rational number.
It follows that $\q+\lin^2=(1/\alpha')(\q'-{\beta'} \lin_0^2)$ is a rational quadratic form, which is a contradiction. 
\end{proof}

\begin{proposition}\label{Prop2.9} Let $G=\SL_n(\RR)$ and $\Gamma=\SL_n(\ZZ)$. Let $(\q, \lin)\in \pairs_n$ and $F$ be a closed subgroup of $G$ for which $\overline{\SO(\q,\lin)^\circ\Gamma}=F\Gamma$. Then
$F\simeq \SL_{n-1}(\RR)\ltimes_l \RR^{n-1}$ if and only if  there exists a non-zero $v\in \QQ^{n}$ that is
$\SO(\q,\lin)$-invariant. 

\end{proposition}
\begin{proof}
Suppose that $F\simeq \SL_{n-1}(\RR)\ltimes_l \RR^{n-1}$. Since $F$ is a $\QQ$-group by Theorem~\ref{Shah}, there is $g_1\in \SL_n(\QQ)$ for which $\SO(\q,\lin)\subseteq g_1^{-1} (\SL_{n-1}(\RR)\ltimes_l \RR^{n-1})g_1$.
Since $\SL_{n-1}(\RR)\ltimes_l \RR^{n-1}$ fixes $e_n$, $\SO(\q,\lin)$ fixes $g_1e_n$ which is a nonzero rational vector.

Conversely, suppose that $\SO(\q, \lin)$ fixes a nonzero rational vector $v\in \QQ^n$. Since
\[
F:=\{g\in \SL_n(\RR) : gv=v \}
\]
is an algebraic group defined over $\QQ$, $F\cap \Gamma$ is a lattice subgroup of $F$. Since $F$ contains $\SO(\q, \lin)$, it follows that $\overline{\SO(\q,\lin)^\circ\Gamma}\subseteq F\Gamma$. Then the equality automatically holds by Theorem~\ref{intermediate subgroup}.
\end{proof}


For closed subgroups $U$, $H$ of $G$, define
\[
X(H,U)=\{g\in G : Ug\subseteq gH \}.
\]

Note that if $g \in X(H, U)$ and $ H\Gamma \subseteq G/\Gamma$ is closed, then the orbit $ U g \Gamma$ is included in the closed subset 
$  gH\Gamma$ and hence cannot be dense. The next theorem asserts that for a fixed $  \varepsilon>0$ and a continuous compactly supported test function $\phi$ by removing finitely many compact subsets  $C_i$ of such sets, the time average over $[0, T]$ of $\phi$ remains within $ \varepsilon$ of the space average for sufficiently large values of $T$.

\begin{theorem}[{\cite[Theorem 3]{DM}}]\label{DM Theorem 3}
Let $G$ be a connected Lie group and $\Gamma$ be a lattice in $G$. Denote by $\mu$ the $G$-invariant probability measure on $G/\Gamma$. Let 
$U= \{ u_t \}$ be an Ad-unipotent one-parameter subgroup of $G$, and let $ \phi: G/\Gamma \to \RR$ be a bounded continuous function. Suppose $\mathcal D$ 
is a compact subset of $G/\Gamma$ and $ \varepsilon>0$. Then there exist finitely many proper closed subgroups $H_1, \dots, H_k$ such that 
$H_i \cap \Gamma$ is a lattice in $H_i$ for all $ 1\le i \le k$, and compact subsets $C_i \subseteq X(H_i, U)$ such that the following holds. For every 
compact subset $\mathcal F \subseteq \mathcal D- \bigcup_{  i=1}^{  k }  C_i \Gamma/\Gamma$ there exists $T_0 \ge 0$ such that for all $x \in {\mathcal{F} }$ and all $T > T_0$ 
we have 
\[ \left|   \frac{1}{T} \int_0^T \phi( u_tx) \ dt - \int_{G/\Gamma} \phi d\mu \right| < \varepsilon. \]
\end{theorem}



If $H$ is isomorphic to $\SO(p,q-1)^\circ$, since we have a classification of all intermediate (connected) Lie subgroups between $\SO(p,q-1)^\circ$ and $\SL_n(\RR)$,
one can obtain concrete statements.
Using Theorem \ref{intermediate subgroup} we will prove Theorem \ref{DM G=SL_n} below, which is in the spirit of  Theorems 4.4 or 4.5 in \cite{EMM}. However, due to the presence of intermediate subgroups, both the statement and the proof are more involved.  

Recall that closed subgroups $H_i$ in Theorem \ref{DM Theorem 3} are those who give the orbit closures of $U$ in $G/\Gamma$. Notice that in our case, since $G=\SL_d(\RR)$ is $\QQ$-algebraic and $\Gamma=\SL_d(\ZZ)$ is an arithmetic lattice subgroup, one can apply Theorem \ref{Shah}, i.e., $H_i$'s are $\QQ$-algebraic and with unipotent radical.

We say that $X \subseteq \RR^d$ is \emph{a real algebraic set} if $X$ is equal to the set of common zeros of a set of polynomials. 
We need the following lemma.

\begin{lemma}\label{finitely many}
Let $X$ be an affine algebraic set over the field of real numbers. 
Suppose that $Y_1, Y_2, \dots$  are countably many affine algebraic sets such that $X$ is covered by the union of $Y_i$, $i \ge 1$. Then $X$ is covered
by the union of only finitely many of $Y_i$. 
\end{lemma}\label{countable}
\begin{proof}
Assume, without loss of generality,  that $X$ is irreducible. The intersection  $X_i = X \cap Y_i$ is an affine algebraic set, 
and hence is either $X$ or a proper algebraic subset of $X$. 
Suppose that there is no $Y_i$ for which $X_i= X$.  
Since every proper algebraic subset is of lower dimension, hence of Lebesgue measure zero, we obtain a contradiction to the assumption that $X$ is covered by countably many $Y_i$'s.
Consequently, there exists $i \ge 1$ such that  $X\subseteq Y_i$.
\end{proof}

\begin{theorem}\label{DM G=SL_n} 
Let $G, \Gamma, H$ and $K$ be as in Section~\ref{notation-groups}. 
Let $\phi$, $\mathcal D$ and $\varepsilon>0$ be as in Theorem \ref{DM Theorem 3}. 
Let $\psi$ be a bounded measurable function on $K$.
Then there exist a finite set $R \subseteq G/\Gamma$
and closed subgroups $L_x \le G$ associated to every $x \in R$ such that
\begin{enumerate}
\item For $x\in R$, $L_x$ is one of followings:
\begin{equation}\label{normalizers SL_n}
\begin{gathered}
TH, \;T\SL_{n-1}(\RR), \;\SO(\q_{\xi})^\circ\;(\xi\in \QQ-\{0\}), \;T\left(\SO(p,q-1)^\circ\ltimes_{u} \RR^{n-1}\right)\\
T\left(\SO(p,q-1)^\circ\ltimes_{l} \RR^{n-1}\right),\;T\left(\SL_{n-1}(\RR)\ltimes_{u} \RR^{n-1}\right) \text{ and }\;T\left(\SL_{n-1}(\RR)\ltimes_{l} \RR^{n-1}\right), 
\end{gathered}
\end{equation}
where $\q_{\xi}$ is defined as in Remark~\ref{intermediate quadratic form}
and $T=\{\diag(e^{t}, \ldots, e^{t}, e^{-(n-1)t}): t\in \RR\}$.
Here, $TL$ is the subgroup generated by $T$ and $L$.

Moreover, for each $x\in R$, $L_x.x\subseteq G/\Gamma$ is a closed submanifold with a positive codimension. In particular, $\mu(L_x.x)=0$.
\item
For every compact set 
$$\mathcal F \subseteq \mathcal D \setminus \bigcup_{x \in R} L_x . x, $$
there exists $t_0>0$ such that for any $x\in \mathcal F$ and every $t>t_0$ the following holds:
\[
\left|
\int_K \phi(a_tkx)\psi(k)dm(k)- \int_{G/\Gamma} \phi\:d\mu \int_K \psi\:dm
\right|\le \varepsilon.
\]
\end{enumerate}
\end{theorem}


\begin{proof}
 We will follow the strategy of Theorem 4.4 (II) in \cite{EMM}. Let us first verify the following statement, which is an analog of \cite[Theorem 4.3]{EMM}:
let $U=\{u_t\}$ be a given Ad-unipotent one-parameter subgroup of $H$.
We need to find sets $R_1$, $R_2$ and closed subgroups $F_x$'s so that for any compact set $\mathcal F\subseteq \mathcal D\setminus \bigcup_{x\in R_1\cup R_2} F_x.x$, there is $T_0>0$ such that for any $x\in \mathcal F$ and $T>T_0$, it holds that

\begin{equation}\label{EMM Theorem 4.3 G=SL_n}
m\left(\left\{ k \in K : \left| \frac 1 T \int_0^T \phi (u_tkx) dt - \int_{G/\Gamma} \phi d\mu\right|>\varepsilon \right\}\right)\le \varepsilon.
\end{equation}

Let $H_i=H_i(\phi, K\mathcal D, \varepsilon)$ and $C_i=C_i(\phi, K\mathcal D, \varepsilon)$, $1\le i \le k$, be as in Theorem \ref{DM Theorem 3} for $U$.
For each $i$, define 
\[
Y_i=\left\{y\in G : Ky \subset X(H_i, U)  \right\}.
\]
The group generated by $ \bigcup_{k\in K} k^{-1}Uk $ is normalized by $U \cup K$. 
Since $K$ is maximal in $H$, we obtain $\left\langle \bigcup_{k\in K} k^{-1}Uk \right\rangle=H$. 
Let $y\in Y_i$. Since $Uky \subseteq ky H_i$ for all $k \in K$, the previous assertion implies that $H \le yH_i y^{-1}$.
 
Note that $H_i$ is a closed subgroup of $G$ defined over $\QQ$ and $H_i\cap\Gamma$ is a lattice in $H_i$. Moreover, the radical of $H_i$ is unipotent by Theorem~\ref{Shah}.  It follows from Theorem \ref{intermediate type I} that  $F_{i,y}:=yH_iy^{-1}$ belongs to the following list:
\begin{equation}\label{candidates SL_n}
\begin{gathered}
H, \;\SL_{n-1}(\RR), \SO(\q_\xi)^\circ \;(\xi\in \QQ-\{0\}), \;\SO(p,q-1)^\circ\ltimes_{u} \RR^{n-1}\\
\SO(p,q-1)^\circ\ltimes_{l} \RR^{n-1},\;\SL_{n-1}(\RR)\ltimes_{u} \RR^{n-1} \text{ and }\;\SL_{n-1}(\RR)\ltimes_{l} \RR^{n-1}. 
\end{gathered}\end{equation}

Note that the only  groups conjugate to each other in the list \eqref{candidates SL_n} are
the ones of the form $  \SO(\q_\xi)$ for $\xi\in \QQ-\{0\}$. 
Based on this fact we will distinguish two cases:

\vspace{5mm}

\noindent
{\it Case I:} $H_i$ is not isomorphic to $  \SO(\q_\xi)$ for any $\xi\in \QQ-\{0\}$.
Consider $y_1,\;y_2\in Y_i$ such that $F_{i,y_1}=F_{i,y_2}=:F_i$, that is, $y_1^{-1}y_2\in \nor_G(F_i)$, where $F_i$ is one of \eqref{candidates SL_n}. Thus $Y_i\Gamma \subseteq \nor_G(F_i)y_1\Gamma$.  
Since $G$ is semisimple, $\nor_G(F_i)$ is a real algebraic group and has finitely many connected components (\cite[Theorem 3]{Whitney}). Moreover, it is easy to check that $T \subseteq \nor_G(F_i) $
and $\nor_G(F_i)^\circ =TF_i$. Hence, all orbits $Y_i\Gamma/\Gamma$ of this form can be covered
by finitely many orbits of $TF_i$.


\vspace{5mm}

\noindent
{\it Case II:}  $H_i$ is isomorphic to $\SO(\q_{\xi})^\circ$ for some $\xi\in \QQ-\{0\}$. We will partition $Y_i$ as
\begin{equation}\label{partition}
Y_i = \bigsqcup_{_{\xi \in \QQ-\{0\} }} (Y_i \cap Z_\xi), 
\end{equation}
where $Z_\xi= \left\{ y \in G : y H_i y^{-1}= \SO(\q_\xi)^\circ \right\}.$ For each $\xi \in \QQ-\{0\}$, if $Z_\xi$ is non-empty, 
then it is a coset of  $\nor_G(\SO(\q_\xi)^\circ)$, and hence is an algebraic set. Note that 
\[
Y_i=\left\{g\in G : Kg \subset X(H_i, U)  \right\}=  \bigcap_{k\in K} \{ g \in G : kg \in X( H_i, U) \}  
=\bigcap_{k\in K}  k^{-1} X(H_i, U).
\]

We claim that $X(H_i, U)$ is an algebraic set, and $Y_i$, being an intersection of algebraic sets, is also algebraic.
The proof of this claim is essentially included in \cite[Proposition 3.2]{DM}:
Write $h=\dim H_i$, and define
\begin{center}
$\rho_{H_i}=\wedge^h(\Ad): G \to \GL(\bigwedge^h \Lie{g}).$
\end{center}
Note that $\rho_{H_i}$ is an algebraic representation of $G$. We also know (see \cite[Proposition 3.2]{DM}) that $g \in X(H_i, U)$ iff $\Lie{U} \subseteq (\Ad g)  (\Lie{H}_i)$, which, in turn, is equivalent to the condition that $ (\rho_{H_i}( g)  p_{H_i})  \wedge w=0$ for all $ w \in \Lie{U}$. This is, clearly, an algebraic condition. 

It follows from  \eqref{partition} and Lemma~\ref{finitely many} that there exists finitely many rational numbers $\xi_1, \dots, \xi_m$ such that 
\[ Y_i \subseteq  \bigcup_{1 \le j \le m} (Y_i \cap Z_{\xi_j}). \]

For each $1\le j\le m$, suppose that $y_1,\;y_2 \in Y_i$ are such that $y_1 H_i y_1^{-1}=y_2 H_i y_2^{-1}=\SO(\q_{\xi_j})^\circ$.
Then $y_1^{-1}y_2 \in \nor_G(\SO(\q_{\xi_j})^\circ)$ and we conclude that
\[
Y_i\Gamma 
= \bigcup_{1\le j\le m} \left\{ y \in Y_i  : y H_i y^{-1}= \SO(\q_{\xi_j})^\circ \right\} \Gamma 
\subseteq \bigcup_{1\le j\le m} \nor_G(\SO(\q_{\xi_j})^\circ)y_{\xi_j} \Gamma 
\]
for some $y_{\xi_j} \in Y_i$. In view of the fact that $\nor_G(\SO(\q_\xi)^\circ)$ is a finite union of right cosets of $\SO(\q_\xi)^\circ$, there exists a finite set $R \subseteq G/\Gamma$ and a closed subgroup $L_x$ as in \eqref{candidates SL_n} so that
\[
\bigcup_{i} Y_i \Gamma/\Gamma \subseteq \bigcup_{x\in R} L_x.x.
\]

By the definition of $H_i$, $L_x.x$ for each $x\in R$ is a proper closed submanifold in $G/\Gamma$.

Since $X(H_i,U)$ is a real analytic submanifold and $K$ is connected, for any $x\in \mathcal F$,
\[
m\left(\left\{k\in K : kx\in \bigcup_{1\le i\le k} C_i\Gamma/\Gamma \right\}\right)=0.
\] 

By Theorem 4.2 in \cite{EMM}, there is an open set $W\subset G/\Gamma$ for which $\bigcup_{1\le i\le k} C_i\Gamma/\Gamma \subseteq W$ and $m\left(\left\{k\in K: kx\in W\right\}\right)<\varepsilon$ for any $x\in \mathcal F$. 

Let $T_0$ be as in Theorem~\ref{DM Theorem 3}. Then for any $x\in \mathcal F$ and $k\in K$ with $kx\notin W$, we have
\[
\left|\frac 1 T \int_0^T \phi(u_tkx)dt - \int_{G/\Gamma} \phi d\mu\right| < \varepsilon,
\]

which shows \eqref{EMM Theorem 4.3 G=SL_n}. We will skip the rest of the proof since it closely parallels the proof of Theorem 4.4 (II) in \cite{EMM} once we replace \cite[Theorem 4.3]{EMM} by inequality \eqref{EMM Theorem 4.3 G=SL_n}.
\end{proof}

\begin{theorem}\label{DM G=proper}
Let $g_0\in \SL_n(\RR)$ be such that $G_1:=g_0^{-1} ( \SL_{n-1}(\RR)\ltimes_l \RR^{n-1} ) g_0$ is defined over $\QQ$ and $\Gamma_1:=G_1\cap \SL_n(\ZZ)$ is a lattice in $G_1$.
Let $H^{g_0}=g_0^{-1} H g_0 < G_1$. Let $K^{g_0}$ be a maximal compact subgroup of $H^{g_0}$ and $\{a^{g_0}_t=g_0^{-1} \diag(e^{-t}, e^t, 1, \ldots, 1) g_0: t\in \RR \}$ be a one-parameter subgroup of $H^{g_0}$.
Let $\phi$, $\mathcal D$ and $\varepsilon>0$ be as in Theorem \ref{DM Theorem 3} for $G=G_1$ and $\Gamma=\Gamma_1$ and let $\psi$ be a bounded measurable function on $K^{g_0}$.
Then there are finitely many points ${x_i}$ and closed subgroups $L_i$, $1\le i \le \ell$, so that $(g_0^{-1} L_ig_0).x_i$ is closed for 
every $1\le i \le \ell$, and 
for any compact 
$\mathcal F \subseteq \left(\mathcal D - \bigcup_{i=1}^{\ell} (g_0^{-1} L_ig_0).x_i\right)$,
there is $t_0>0$ such that for any $x\in \mathcal F$ and $t>t_0$,
\[
\left|
\int_{K^{g_0}} \phi(a^{g_0}_tkx)\psi(k)dm(k)- \int_{G_1/\Gamma_1} \phi\:d\mu \int_{K^{g_0}} \psi\:dm
\right|\le \varepsilon.
\]

Here, $L_i$ is one of 
\begin{equation}\label{candidates}
\SO(p,q-1)^\circ, \;\SL_{n-1}(\RR), \text{ and }\;\SO(p,q-1)^\circ\ltimes_{l} \RR^{n-1}.
\end{equation}
\end{theorem}

\begin{proof}
The proof is similar to that of Theorem~\ref{DM G=SL_n}. In this case, possible proper intermediate subgroups $H_i$'s are listed in \eqref{candidates}. It is not hard to see that for each $H_i$ in this list, $H_i= \nor_G( H_i)^{\circ}$.
\end{proof}

%
%
%

\section{Siegel integral formula for an intermediate subgroup}\label{section:alpha}
In this section, we will prove a version of Siegel's integral formula for intermediate subgroups $F$ in Proposition~\ref{Prop2.9}. For a 
bounded and compactly supported function $f:\RR^n\rightarrow \RR$, the Siegel 
transform of $f$ is defined by
\[  \widetilde{f} (g) := \widetilde{f} (g \ZZ^n) = \sum_{_{v \in \ZZ^n - \{ \origin \} }} f( gv). \]

\begin{lemma}[Schmidt, {\cite[Lemma 3.1]{EMM}}]\label{lem:Schmidt}
Let $f: \RR^n \to \RR$ be a bounded function vanishing outside of a bounded set. Then there exists a constant $c=c(f)$ such that
\[ \widetilde{ f}( \Lambda) < c \alpha( \Lambda) \]
for all unimodular lattices $ \Lambda$ in $\RR^n$. 
\end{lemma}

In the rest of this section we will change the notation slightly and write $ \alpha(g)$ for $ \alpha(g \ZZ^n)$.  
One can see that the inequality  $\alpha(g_1g_2)\le \alpha(g_1)\alpha(g_2)$ does not always hold. The following lemma singles out special cases 
in which this inequality holds.

\begin{lemma}\label{lemma:upperbound of alpha}
Let $a, g, g_1, g_2\in \SL_n(\RR)$. Assume, further, that $a$ is self-adjoint. Then we have
\begin{enumerate}
\item $\dfrac{ \alpha(g_1g_2)}{ \alpha(g_2)}\le \max\limits_{1\le j\le n} \|\wedge^{j}g_1^{-1}\|_{\op}$. 

\vspace{0.05in}
\item $\alpha(ag)\le \alpha(a)\alpha(g)$.
\end{enumerate}
\end{lemma}

\begin{proof}
By the definition of $ \alpha$ we have 
\[
\alpha(g_1g_2)=\max_{1\le j \le n}
\left\{\frac 1 {\|g_1\ov_1 \wedge \cdots \wedge g_1\ov_j\|} : \begin{array}{l}
\ov_1, \ldots, \ov_j \in g_2\ZZ^n,\\
\ov_1 \wedge \cdots \wedge \ov_j \neq 0  \end{array}
\right\}.
\]

It follows from the definition of the operator norm that for any $1\le j\le n$ and any linearly independent vectors $\ov_1, \ldots, \ov_j \in g_2\ZZ^n$, we have
\begin{equation*}
\|g_1\ov_1 \wedge \cdots \wedge g_1\ov_j\|
=\|(\wedge^i g_1)(\ov_1 \wedge \cdots \wedge \ov_j)\|
\ge \|\wedge^i g_1^{-1}\|^{-1}_{\op} \:\|\ov_1 \wedge \cdots \wedge \ov_j\|.
\end{equation*}
This proves (1). 

\vspace{0.1in}
In order to show (2), we first assume that $a=\diag(a_1,\ldots, a_n)$.  
Recall that for multi-indices $I=\{1\le i_1 < \ldots < i_j\le n\}$ and $L=\{1\le \ell_1 < \ldots < \ell_j \le n\}$,
the $(I,L)$-component of $\wedge^i a$ is
\[(\wedge^i a)_{IL} = \left\{\begin{array}{cl}
\prod_{j} a_{i_j}, &\text{if } I=L;\\
0, &\text{otherwise.}\end{array}\right.
\]

Therefore 
\begin{equation}\label{eqn:upperbound of alpha 2}
\sup_{1\le j\le n} \| \wedge^j a^{-1}\|_{\op} 
=\sup_{1 \le j \le n} \left(\sup\left\{\frac 1 {a_{i_1}\cdots a_{i_j}} : 1\le i_1 < \ldots < i_j \le n \right\}\right).
\end{equation}

On the other hand, since $a$ is a diagonal matrix, 
\begin{equation}\label{eqn:upperbound of alpha 3}
\alpha(a)=\sup_j \left(\frac 1 {\min\left\{{a_{i_1}\cdots a_{i_j}} : 1 \le i_1 < \ldots < i_j \le n\right\}}\right).
\end{equation}

Combining \eqref{eqn:upperbound of alpha 2} and \eqref{eqn:upperbound of alpha 3} with the first result, we obtain the second property.

For an adjoint matrix $a'\in\SL_n(\RR)$, we can write $a'=kak^{-1}$, where $a$ is diagonal and $k\in \SO(n)$. 
Notice that the $\alpha$ function is invariant under left multiplication by $\SO(n)$.
Using (2),
\[
\alpha(a'g)=\alpha((kak^{-1})g)=\alpha(ak^{-1}g)
\le\alpha(a)\alpha(k^{-1}g)=\alpha(a')\alpha(g).
\]
\end{proof}


The following theorem is an analog of Lemma 3.10 in \cite{EMM}, where a similar statement for the integral of 
$ \alpha^r$ over $\X_n$ is proven.

\begin{theorem}\label{upperbound of alpha}
Let $g_0 \in \SL_n(\RR)$ be such that the algebraic group $$F=g_0^{-1}\left(\SL_{n-1}(\RR)\lt\RR^{n-1}\right)g_0$$ is defined over $\QQ$ 
and that $\Gamma_F:=F \cap \Gamma$ is a lattice in $F$. Denote by $\mu_F$ the $F$-invariant probability measure on $F/\Gamma_F$, 
and let $\mathcal F_{F} \subseteq F$ be a fundamental domain for the action of $\Gamma_F$ on $F$. 
Then for any $1\le r <n-1$,
\[
\int_{\mathcal F_{F}} \alpha^r(g) d\mu^{}_F(g) < \infty.
\]
\end{theorem}

\begin{proof} 
Since $F$ is defined over $\QQ$, there exists $g_1\in \SL_n(\RR)$ such that
$g_0^{-1}\left(\SL_{n-1}(\RR)\lt\RR^{n-1}\right)g_0=
g_1^{-1}\left(\SL_{n-1}(\RR)\lt\RR^{n-1}\right)g_1$ and
$F_0=g_1^{-1}\SL_{n-1}(\RR)g_1$ is a Levi subgroup for $F$ defined over $\QQ$. Note that the unipotent radical of $F$ is given by $R=g_1^{-1}(\{\Id_{n-1}\}\ltimes \RR^{n-1})g_1$ and is defined over $\QQ$
(see \cite{Borel}). 

Recall that if $H$ is a connected algebraic group defined over $\QQ$, then the discrete subgroup $H(\ZZ)$ is a lattice in $H$ if and only if $H$ does not admit a nontrivial character defined over $\QQ$ (see \cite[Theorem 4.13]{PR}).
Since $F_0$ is semisimple and $R$ is polynomially isomorphic to $\RR^{n-1}$, they do not have nontrivial polynomial characters, hence $F_0(\ZZ)=F_0 \cap \Gamma_F$ and $R(\ZZ)=R\cap \Gamma_F$ are lattices in $F_0$ and $R$, respectively.
Moreover, since $R$ is abelian, $R(\ZZ)$ is cocompact. 

Let $\mathcal F_{F_0}$ and $\mathcal F_R$ be fundamental domains for $F_0/F_0(\ZZ)$ and $R/R(\ZZ)$, respectively. 
One can find a fundamental domain $\mathcal F_{F}\subseteq \mathcal F_{F_0} \times \mathcal F_{R}$.

Now, we want to cover $\mathcal F_{F_0}$ by a finite union of copies of a Siegel set of $\SL_{n-1}(\RR)$.
Recall that the standard Siegel set $\Sigma=\Sigma_{\eta,\xi}$ of $\SL_n(\RR)$ is the product $\SO(n) A_{\eta} N_{\xi}$, where
\[\begin{split}
A_\eta&=\left\{\diag(a_1, \ldots, a_n)\in \SL_n(\RR) : 0< a_i<\eta a_{i+1} \right\}\;\text{and}\\
N_{\xi}&=\left\{(u_{ij})\text{: upper unipotent}\in \SL_n(\RR) : |u_{ij}| \le \xi\right\}. 
\end{split}\]
It is well-known that a fundamental domain of $\SL_n(\RR)/\SL_n(\ZZ)$ is contained in $\Sigma_{\eta,\xi}$ for some appropriate $\eta,\; \xi>0$ (see \cite[Theorem 4.4]{PR} for instance). 
Moreover, since $F_0$ is a semisimple Lie group defined over $\QQ$ and $g_1 F_0 g_1^{-1}$ is self-adjoint,
by a theorem of Borel and Harish-Chandra (\cite{BH}, see also \cite[Theorem 4.5 and Theorem 4.8]{PR}),
there are $\gamma_1, \ldots, \gamma_k\in \SL_n(\ZZ)$ such that 
for $\mathcal{D}=\left(\bigcup_{i=1}^k g_1^{-1} \Sigma \gamma_i\right)\cap F_0$, one has $\mathcal{D} F_0(\ZZ)=F_0$.

Note that $g_1^{-1}\Sigma g_1$ is a Siegel set with respect to the Iwasawa decomposition $K^{g_1}=g_1^{-1}K_0g_1$, $A^{g_1}=g_1^{-1}A_0g_1$ and $N^{g_1}=g_1^{-1}N_0g_1$. 
By \cite[Lemma 7.5]{BH}, for each $g_1^{-1}\Sigma \gamma_i=g_1^{-1}\Sigma g_1 (g_1^{-1}\gamma_i)$, 
there are finitely many $g^i_j$'s for which 
\[
g_1^{-1}\Sigma\gamma_i\cap F_0 \subseteq \bigcup_j g_1^{-1}\Sigma_1 g_1 g^i_j,
\]
for some $\Sigma_1$, where $\Sigma_1$ is some standard Siegel set of $\SL_{n-1}(\RR)(\subseteq \SL_n(\RR))$, so that $g_1^{-1}\Sigma_1 g_1$ is a Siegel set with respect to the Iwasawa decomposition $K^{g_1}\cap F_0$, $A^{g_1}\cap F_0$ and $N^{g_1}\cap F_0$.
Therefore, by change of variables and using the fact that $\SL_n(\RR)$ is unimodular,
\[
\int_{\mathcal F_{F}} \alpha^r(g) d\mu^{}_{\mathcal F}(g)
\le \sum_{i,j} \int_{\Sigma_1\times \mathcal F_R} \alpha^r(g_1^{-1}g g_1g^{i}_j \:h) d\mu^{}_{\SL_{n-1}(\RR)}(g) d\mu^{}_R(h).
\]

Let $\Sigma_1=(\Sigma_1)_{\eta',\xi'}$ and denote $g=k' a' n'$, where $k'\in \SO(n-1)$, $a'=\diag(a'_1, \ldots, a'_{n-1}, 1)$ for which $a'_i \le \eta' a'_{i+1}$ and $n'=(u'_{ij})$ is the upper unipotent element in $\SL_{n-1}(\RR)\ltimes \{0\}$ such that $|u'_{ij}|\le \xi'$ for any $(i,j)$ with $i<j$.
Since $d\mu^{}_{\SL_{n-1}(\RR)\ltimes \{0\}}$ is locally $\Delta(a') dk'da'dn'$, where $\Delta(a')$ is the product of positive roots, using Lemma \ref{lemma:upperbound of alpha} and Lemma 3.10 in \cite{EMM}, it follows that for $1\le r < n-1$,
\[\begin{split}
\int_{\mathcal F_{F}} \alpha^r(g) d\mu^{}_{\mathcal F}(g)
&\ll_{g_1}
\sum_{i,j} \int_{A'_{\eta'}}\int_{N'_{\xi'}\times\mathcal F_R} \alpha^r(a')\alpha^r(n' g_1g^i_j h) \Delta(a') da'dn' d\mu^{}_R(h)\\
&\le C \sum_{i,j} \int_{A'_{\eta'}} \alpha^r(a') \Delta(a')da' < \infty
\end{split}\]
for some $C>0$ since $N_{\xi'}\times \mathcal F_{R}$ is compact. 
Here, 
\[\begin{split}
A'_{\eta'}&=\left\{\diag(a_1, \ldots, a_{n-1}, 1) \in \SL_{n-1}(\RR) : 0<a_i\le \eta' a_{i+1} \right\}\;\text{and}\\
N'_{\xi'}&=\left\{(u'_{ij})\text{: upper unipotent}\in \SL_{n-1}(\RR) : |u'_{ij}| \le \xi'\right\}.
\end{split}\]
\end{proof}


Recall the well known Siegel's integral formula. 

\begin{theorem}[Siegel, \cite{Siegel}]\label{Siegel integral formula} For a bounded and compactly supported function $f:\RR^n\rightarrow \RR$, we have
\[
\int_{G/\Gamma} \widetilde{f} (g) d\mu(g)= \int_{\RR^n} f(v)dv.
\]
\end{theorem}

We also need the analog of Siegel's integral formula for the following specific intermediate subgroup.

\begin{theorem}\label{Siegel integral formula:proper F}
Assume that $g_0\in \SL_n(\RR)$ is such that $F=g_0^{-1}\left(\SL_{n-1}(\RR)\lt\RR^{n-1}\right)g_0$
is an algebraic group defined over $\QQ$ and that $\Gamma_F:=F \cap \Gamma$ is a lattice. Denote by $\mu_F$ the probability $F$-invariant measure on $F/\Gamma_F$
and by $\mathcal F_F$ a fundamental domain for $\Gamma_F$ in $F$. 
Then for any bounded compactly supported measurable function $f:\RR^n\rightarrow \RR$, we have 
\[
\int_{\mathcal F_F} \widetilde f(g) d\mu^{}_F(g)
=\int_{\RR^n} f(\ov) d\ov+ \sum_{m\in \ZZ-\{0\}} f(m\:k_0g_0^{-1}e_n), 
\]
where  $k_0$ is determined by $\RR.g_0^{-1}e_n \cap \ZZ^n=\ZZ.k_0g_0^{-1}e_n$. 
\end{theorem}

\begin{proof}

By Lemma \ref{lem:Schmidt} and Theorem \ref{upperbound of alpha}, the integral 
$$\int_{F/\Gamma_F} \widetilde f(g) d\mu^{}_{F}(g)$$ 
is finite, and the map sending $f$ to $ \int_{F/\Gamma_F} \widetilde f(g) d\mu^{}_{F}(g)$ is a continuous positive linear functional on the space of compactly 
supported continuous functions and is hence given by a finite measure. 

Note that the set of $F$-fixed vectors in $\RR^n$ is $\RR.g_0^{-1}e_n$ which is defined over $\QQ$, and $F$ acts transitively on $\RR^n-\RR.g_0^{-1}e_n$. Since $\RR.g_0^{-1}e_n \cap g\ZZ^n$ is $\ZZ$-span of  $k_0g_0^{-1}e_n$ for some $0\neq k_0\in \RR$, it follows from the usual argument of Siegel's integration formula combined with Proposition~\ref{upperbound of alpha} (see \cite[Section 3]{Han21}) that
\[
\int_{\mathcal F_F} \tilde f (g) d\mu(g)= \int_{\RR^n} f(\ov) d\ov + \sum_{m\in \ZZ-\{0\}} f(m\:k_0g_0^{-1}e_n).
\]
\end{proof}


\section{Upper bounds for spherical averages of the $\alpha$-function }\label{sec:bounds}
In this section we will prove the following theorem, which is an analog of \cite[Theorem 3.2]{EMM}.

\begin{theorem}\label{alphabounded2}
\begin{enumerate}
\item For $p\ge 3$, $q\ge 2$ and $0<s<2$. Then for every $g \in \SL_n(\RR)$ we have 
\[
\sup_{t>0} \int_K \alpha(a_tk.g\ZZ^n)^s dm(k) < \infty.
\]
\item For $p=2$, $q=3$, there is $0<s<1$ such that
\[
\sup_{t>0} \int_K \alpha(a_tk.g\ZZ^n)^s dm(k) < \infty.
\]
\end{enumerate}
\end{theorem}

The proof is based on the following proposition, which is Proposition 5.12 of \cite{EMM}.

\begin{proposition}\label{EMM98 Proposition 5.12} Consider a self-adjoint reductive subgroup $H$ of $\GL_n(\RR)$. Let $K=\O_n(\RR)\cap H$ and let $m$ be the normalized Haar measure of $K$. Let $A=\{a_t : t\in \RR\}$ be a self-adjoint one-parameter subgroup of $H$ and let $\mathcal F$ be a family of strictly positive function on $H$ having the following properties:
\begin{enumerate}
\item[(a)] For any $\vep>0$, there is a neighborhood $V(\vep)$ of $\Id$ in $H$ such that for any $f\in \mathcal F$,
\[
(1-\vep)f(h) < f(uh) < (1+\vep)f(h),\; \forall h\in H,\; \forall u\in V(\vep).
\]
\item[(b)] For any $f\in \mathcal F$, $f(Kh)=f(h)$, for all $h \in H$.
\item[(c)] $\sup_{f\in \mathcal F} f(\Id)<\infty$.
\end{enumerate}

Then there exists a positive constant $c=c(\mathcal F) <1$ such that for all $t_0 >0, b>0$, there exists $B=B(t_0,b)<\infty$ with the following property:
if $f \in \mathcal F$ and 
\begin{equation}\label{condition (d)} 
\int_K f(a_{t_0} k h) dm(k) < c f(h) + b 
\end{equation}
for $\forall h \in KAK \subset H,$
then $$\int_K f(a_\tau k) dm(k) <B$$ for any $\tau>0.$

\end{proposition}

\subsection{Reduction to an $f_\vep$-function} Let us start by recalling the definition of Benoist-Quint function.
Write $\bigwedge(\RR^n)=\bigoplus_{i=1}^{n-1} \bigwedge^i(\RR^n)$, and consider the representation $\rho : H \rightarrow \GL(\bigwedge(\RR^n))$ induced by the linear representation of $H$ on $\RR^n$.  
Since $H$ is semisimple, $\rho$ decomposes into a direct sum of irreducible representations. For each highest weight $\lambda$, denote by $V^{\lambda}$ the direct sum of all irreducible components with highest weight $\lambda$ and by $\tau_{\lambda}$ the orthogonal projection on $V^{\lambda}$.

For $\vep>0$ and $0<i<n$, we define the Benoist-Quint $\varphi$-function $\varphi_{\vep} : \bigwedge(\RR^n)\rightarrow [0, \infty]$ as in \cite{BQ}, \cite{Sar}:
$$\varphi_{\vep} (v) = \left\{\begin{array}{ll}
    \min_{\lambda\neq 0} \vep^{(n-i)i}\|\tau_{\lambda}(v)\|^{-1}, & \hbox{if $\|\tau_0(v)\|\leq \vep^{(n-i)i}$;} \\
    0, & \hbox{otherwise.}
  \end{array}\right.
$$
Note that $V^0= \{v \in \bigwedge(\RR^d) : Hv=v \}$ by definition. Denote by $(V^0)^\perp$ its orthogonal complement in $\bigwedge(\RR^d).$ 

\begin{remark}\label{norm}
\begin{enumerate}
\item Since $\tau_{\lambda}$ is defined in terms of projection of $v$ onto $V^{\lambda}$, for every $v \in \bigwedge(\RR^d)$ and $\lambda\neq 0$, we have
$$\tau_{\lambda}(v)= \tau_{\lambda}(v- \tau_0(v)). $$
\item Since $\max_{\lambda \neq 0} \| \tau_\lambda (v) \|$ defines a norm on $(V^0)^\perp$, there exists $c_1 >1$ such that for all $v \in (V^0)^\perp,$
\begin{equation}\label{norm-eq}
\frac{1}{c_1\|v \|} 
\leq \frac 1 {\max_{\lambda\neq 0} \|\tau_\lambda(v)\|}
\leq c_1 \frac{1}{\|v \|}.
\end{equation}

\end{enumerate}
\end{remark}

\subsection{The function $f_{\varepsilon}$ and associated inequalities}
Recall that $\Omega(\Lambda)=\bigcup_{i=1}^n \Omega^i(\Lambda)$, 
where $\Omega^i(\Lambda)$ is defined by
\[
\Omega^i(\Lambda)=\left\{v=v_1\wedge \cdots \wedge v_i: v_1, \ldots, v_i \in \Lambda\right\} \setminus \{ 0 \}.
\]

For $ \varepsilon>0$, define  $f_{\vep}: \SL_n(\RR)/\SL_n(\ZZ) \to [0, \infty]$ by
$$f_\vep(\Lambda)=  \max_{v\in \Omega(\Lambda)} \varphi_{\vep} (v).$$
We will first show that although $f_\vep$ is not finite on its entire domain, its restriction to each $H$-orbit $H.\Lambda$ is finite for sufficiently small $\varepsilon.$

\begin{lemma}\label{claim1} For a given $g\in \SL_n(\RR)$, there is $\vep_0>0$ such that if $0< \vep < \vep_0$, the function $$f_{g,\vep}(h):=f_{\vep}(hg\ZZ^n)$$ has a finite value for all $h\in H$.
\end{lemma}
\begin{proof}
Observe that $f_{\vep}(hg\ZZ^n)=\infty$ if and only if there is $1<i<n$ and  $0\neq v \in \Omega^i(g\ZZ^n)\cap V^0$ for which $\|v\|\le \vep^{i(n-i)}$. 
Since any element in $H$ is of the form $\diag(M, 1)$ with $M \in \SO(p,q-1)^\circ$, $H$ acts on $\bigoplus_{i=1}^{n-1} \RR.e_i$ 
irreducibly. This implies that any nonzero $H$-fixed elements $v\in \Omega(g\ZZ^n)$ are scalar multiples of $e_n$, $e_1\wedge \cdots \wedge e_{n-1}$, or $e_1\wedge \cdots \wedge e_n$.
If $ \Omega(g\ZZ^n)$ does not contain any such vectors other than $e_1\wedge \cdots \wedge e_n$, any value of $ \varepsilon>0$ will work. 
Otherwise, there exists a non-empty set $S$ of vectors $v\in \Omega(g\ZZ^n)$ which are of the form 
\[ v= a(v) e_n \quad \textrm{or} \quad v= a(v)(e_1\wedge \cdots \wedge e_{n-1}) \]
for some $a(v) >0$. Since $g\ZZ^n$ is discrete,  $ \vep_0:= \min \{ a(v)^{1/(n-1)}: v \in S \}>0$.
If $\varepsilon<\varepsilon_0$, there are no vectors in $\Omega(g\ZZ^n)\cap V^0$ or norm at most $\varepsilon^{i(n-i)}$.
It follows that the restriction of $f_\varepsilon$ to $Hg\ZZ^n$ is finite. 
\end{proof}

\begin{lemma} \label{claim2} Let $s>0$ and $g\in \SL_n(\RR)$. Let $\vep>0$ be such that $f_{g,\vep}(h)<\infty$ for all $h \in H$.  Then there exist $c_{s,\vep}>0$ and $C_{s,\vep}>0$ depending on $s$ and $\vep$, such that for all $ h \in H$ we have
\[
\alpha(hg\ZZ^n)^s \le c_{s,\vep} f_{g,\vep}(h)^s + C_{s,\vep}.
\]
\end{lemma}

\begin{proof} Write $ \vep_1=\min_{1\le i\le n-1} \vep^{i(n-i)}$ and $\vep_2= \max_{1\le i\le n-1} \vep^{i(n-i)}$, and define
$c_{s,\vep}=\left({c_1}/ \varepsilon_1\right)^s$ and $C_{s,\vep}= \varepsilon_2^s+1$, where $c_1$ is chosen as in \eqref{norm-eq}.
In view of \eqref{norm-eq}, for all $v \in \bigoplus_{\lambda\neq 0} V^\lambda$ we have
\[
\frac { \varepsilon_1} {c_1 \|v\|} \le \varphi_{\vep}(v) \le \frac {c_1 \varepsilon_2}{\|v\|}.
\]
Let $v\in \Omega^i(hg\ZZ^n)$ be the vector at which $\alpha(hg\ZZ^n)$ is attained. We will consider two cases. 

If $\|\tau_0(v)\|> \vep^{i(n-i)}$, then we have
\[
\alpha(hg\ZZ^n) = \frac 1 {\|v\|} \le \frac 1 {\|\tau_0(v)\|} \le \vep^{-i(n-i)} \le C_{s, \vep}. 
\]
Otherwise, we have $\|\tau_0(v)\| \le \vep^{i(n-i)}$. In this case, by the choice of $\vep$, we must have $v\neq \tau_0(v)$. This implies that 
\[
\alpha(hg\ZZ^n)= \frac 1 {\|v\|} \le \frac 1 {\|v-\tau_0(v)\|} 
\le \frac{c_1}{ \varepsilon_1} \varphi_{\vep} (v-\tau_0(v))=\frac{c_1}{ \varepsilon_1} \varphi_{\vep}(v) \le c_{s,\vep} f_{g, \vep}(h). 
\]
The claim follows by combining these two cases. 
 \end{proof}




\begin{lemma} \label{claim3} 
Suppose $p\ge 3$, $q\ge 2$, and $ s \in (0,2)$ or $p=2$, $q=2, 3$ and $s \in (0,1)$. Then, for every $c>0$, there exists $t_0>0$ such that for every $t>t_0$ and $v\in \bigwedge^i(\RR^n)-V^0$, the following holds:
\[
\int_K \frac 1 {\max_{\lambda\neq0} \|\tau_\lambda(a_tkv)\|^s} \, dm(k) \le   \frac {c} {\max_{\lambda\neq0} \|\tau_\lambda(v)\|^s}.
\]
\end{lemma}
\begin{proof}
Let $v \in \bigwedge^i(\RR^n)- V^0$.  By part (1) of Remark \ref{norm} we may assume that $v\in\bigoplus_{\lambda \neq 0} V^\lambda$. 
It follows from Proposition 5.4 of \cite{EMM} and the inequality \eqref{norm-eq} that
$$ \int_K \frac 1 {\max_{\lambda\neq0} \|\tau_\lambda(a_tkv)\|^s} dm(k) \le c_1\int_K \frac{1}{\| a_t k v \|^s} dm(k) < c_1c' \frac{1}{\|v\|^s} \le c_1^2c' \frac 1 {\max_{\lambda\neq0} \|\tau_\lambda(v)\|^s}.$$
Indeed, one can use Proposition 5.4 of \cite{EMM} for as follows:
let $W^-, W^0, W^+$ be the eigenspaces corresponding to eigenvalues $e^{-t}, 1, e^t$ (of $a_t$) in $\bigwedge^i (\RR^n)$, respectively.
From $v \not\in V^0$, it follows that $Kv \nsubseteq W^0$. Since $p\ge 3$ and $q\ge 2$, we deduce that conditions (a), (b), (c) of Lemma 5.2 of \cite{EMM} are satisfied.
For $p=2$ and $q=2, \; 3$, one can directly show that conditions (a), (b) of Lemma 5.1 of \cite{EMM} are satisfied.
\end{proof}

\begin{proposition} \label{claim5} Let $g\in \SL_d(\RR)$. Suppose $p\ge 3$, $q\ge 2$, and $ s \in (0,2)$ or $p=2$, $q=2, 3$ and $s \in (0,1)$.  
One can find $\varepsilon_1>0$ for which for any $\varepsilon\in (0, \vep_1)$ and for any $c>0$, there are $t_0$ and $b>0$ such that for every $h \in H$ the following inequality holds:
$$ \int_K f_{g, \varepsilon} (a_{t_0} kh)^s dm(k) < c f_{g, \varepsilon}(h)^s + b.$$
\end{proposition}

\begin{proof}
Let $\Omega^i$ be the set of monomials in $\bigwedge^i(\RR^n)$ for $0\le i\le n$.
By Lemma 3.4 of \cite{Sar}, there exists $C>0$ such that for all $0<\varepsilon<1/C,$ and $u \in\Omega^{i_1}$, 
$v \in \Omega^{i_2}, w \in \Omega^{i_3}$ with $i_1 \geq 0, i_2 >0, i_3>0$ and 
$$\varphi_\varepsilon ( u \wedge v ) \geq 1, \varphi_\varepsilon (u \wedge w ) \geq 1,$$ we have following:
\begin{enumerate}
\item If $i_1 >0$ and $i_1 + i_2 + i_3 <d,$ then
$$ \min \{ \varphi_\varepsilon ( u \wedge v), \varphi_\varepsilon ( u \wedge w) \} \leq (C \varepsilon )^{\frac{1}{2}} \max \{\varphi_\varepsilon(u),  \varphi_\varepsilon ( u \wedge v \wedge w) \}.$$
\item If $i_1=0$ and $i_1 + i_2+ i_3 < d$, then
$$\min \{ \varphi_\varepsilon ( v ) , \varphi_\varepsilon ( w ) \} \leq (C \varepsilon)^{\frac{1}{2}} \varphi_\varepsilon (v \wedge w).$$
\item If $i_1 >0, i_1 + i_2 + i_3 = d$ and $\| u \wedge v \wedge w \| \geq 1,$ then
$$\min \{ \varphi_\varepsilon ( u \wedge v) , \varphi_\varepsilon ( u \wedge w) \} \leq (C\varepsilon)^{\frac{1}{2}} \varphi_\varepsilon(u).$$
\item If $i =0, i_1 + i_2 + i_3 =d$ and $\|v \wedge w \| \geq 1,$ then
$$\min \{ \varphi_\varepsilon(v), \varphi_\varepsilon (w) \} \leq b_1,$$
where $b_1 = \sup \{ \varphi_\varepsilon (v) : v \in \bigwedge (\RR^n): \|v\| \geq 1 \}.$
\end{enumerate}

By Lemma~\ref{claim3}, there exists $t_0>0$, independent of the choice of $\varepsilon>0$, such that for any $v\in \bigwedge(\RR^n)$ with $\varphi_\vep(v)\neq 0$, we have
\begin{equation}\label{eq claim3}
\int_K \varphi_\varepsilon (a_{t_0}kv)^s dm(k) \le \frac c {2n} \varphi_\varepsilon (v)^s.
\end{equation}
Let $m_0=e^{t_0s}\ge 1$ so that
\[
\frac 1 {m_0} \varphi_\varepsilon (v)
\le \varphi_\varepsilon (a_{t_0}v)
\le m_0\varphi_\varepsilon (v).
\]
Define the set 
\[
\Psi(hg\ZZ^n)=\left\{v\in \Omega(hg\ZZ^n): f_{\varepsilon,g}(h) \le m_0^2 \varphi_\varepsilon(v)\right\}.
\]
Note that 
\[f_{\varepsilon,g}(h)
=\max_{v\in \Omega(hg\ZZ^n)} \varphi_\varepsilon(v)
=\max_{v\in \Psi(hg\ZZ^n)} \varphi_\varepsilon(v)
\] 
and if $v\in \Omega(hg\ZZ^n)$ is such that $f_{g, \varepsilon}(h)=\varphi_\varepsilon(v)$, then $v\in \Psi(hg\ZZ^n)$. 
Choose $\varepsilon>0$ small enough so that 
\begin{equation}\label{condition m_0 C}
m_0^4 C \varepsilon <1.
\end{equation}

\vspace{0.1in}
\noindent Case 1. $f_{\varepsilon,g}(h)=f_{\varepsilon}(hg\ZZ^n)\le \max\{b_1, m_0^2\}$.

For any $k\in K$, since $f_\varepsilon$ is left $K$-invariant,
\[
f_\varepsilon(a_{t_0}khg\ZZ^n)
\le m_0 f_\varepsilon(khg\ZZ^n)
=m_0 f_\varepsilon(hg\ZZ^n),
\]
hence it follows that
\begin{equation}\label{eq Case 1}
\int_K f_{g, \varepsilon}(a_{t_0}kh)^s dm(k)
\le \left( m_0\max\left\{b_1, m_0^2\right\}\right)^s.
\end{equation}

\vspace{0.1in}
\noindent Case 2. $f_{\varepsilon,g}(h)> \max\{b_1, m_0^2\}$.

One can deduce that $\Psi(hg\ZZ^n)$ contains at most one element up to sign change in each degree from the exactly same argument with Claim 3.9 in \cite{Sar} with the assumption $m^4_0C\varepsilon<1$.

Note that for any $v\in \Omega(hg\ZZ^n)$, 
\[\varphi_\varepsilon(a_{t_0}kv) \le \max_{\psi\in \Psi(hg\ZZ^n)} \varphi_\varepsilon(a_{t_0}k\psi)\]
since if $v\in \Psi(hg\ZZ^n)$, it is obvious and if $v\notin \Psi(hg\ZZ^n)$, by the definition of $\Psi(hg\ZZ^n)$ and $m_0$,
\[
\varphi_\varepsilon(a_{t_0}kv)
\le m_0 \varphi_\varepsilon(v)
\le m_0^{-1} f_\varepsilon(hg\ZZ^n)
\le m_0^{-1} \max_{\psi\in \Psi(hg\ZZ^n)} \varphi_\varepsilon(\psi)
\le \max_{\psi\in \Psi(hg\ZZ^n)} \varphi_\varepsilon(a_{t_0}k\psi).
\]
Hence
\[
\int_{K} f_{\varepsilon} (a_{t_0} khg\ZZ^n)^s dm(k)
\le \sum_{\psi \in \Psi(hg\ZZ^n)} \int_K \varphi_\varepsilon(a_{t_0}k\psi)^s dm(k).
\]
For any $\psi\in \Psi(hg\ZZ^n)$, $0<f_{\varepsilon,g}(h)/m_0^2\le \varphi_\varepsilon(v)$, we have $\psi\notin V^0$, hence by \eqref{eq claim3},
\[
\int_K \varphi_\varepsilon(a_{t_0}k\psi)^s dm(k) \le \frac c {2n} \varphi_\varepsilon (\psi)^s.
\]
Since there is at most $2n$ elements in $\Psi(hg\ZZ^n)$,
\begin{equation}\label{eq Case 2}
\int_{K} f_{\varepsilon} (a_{t_0} khg\ZZ^n)^s dm(k)
\le c\max_{\psi\in \Psi(hg\ZZ^n)} \varphi_\varepsilon(\psi)^s=cf_{\varepsilon,g}(h)^s.
\end{equation}

Therefore, by \eqref{eq Case 1} and \eqref{eq Case 2}, it follows that
\[
\int_K f_{\varepsilon,g}(a_{t_0}kh)^s dm(k)
\le c f_{\varepsilon,g}(h)^s+\left(m_0\max\{b_1, m_0^2\}\right)^s.
\]
\end{proof}

\begin{proof}[Proof of Theorem~\ref{alphabounded2}]
By Lemma~\ref{claim2}, it suffices to show that
\[
\sup_{t>0} \int_K f_{g,\vep}(a_tk.g\ZZ^n)^s dm(k)<\infty
\]
for an appropriate $\varepsilon>0$, using Proposition~\ref{EMM98 Proposition 5.12}. 
The assumptions of Proposition~\ref{EMM98 Proposition 5.12} are obvious except the condition (c) and the inequality \eqref{condition (d)}.
Choose $\varepsilon>0$ such that Lemma~\ref{claim1} and the inequality \eqref{condition m_0 C} holds. Note that $m_0\ge 1$ in \eqref{condition m_0 C} is determined once $0<c<1$ in Proposition~\ref{EMM98 Proposition 5.12} is given.
Then Lemma~\ref{claim1} shows the condition (c) and Proposition~\ref{claim5} shows the inequality \eqref{condition (d)}.
\end{proof}

\begin{proof}[Proof of Theorem \ref{main2}]
The proof works exactly as the proof of Theorem 3.4 in \cite{EMM}. Instead of using 
Theorem 3.2 in \cite{EMM}, one needs to use Theorem \ref{alphabounded2} and
one of Theorem \ref{DM G=SL_n} and Theorem \ref{DM G=proper} depending on the orbit closure. 

\end{proof}


\section{Passage to dynamics on the space $\X_n$ of unimodular lattices in $\RR^n$}\label{sec:bridge}
In this section we will show how to use the equidistribution results of previous 
sections to prove Theorem \ref{thm:main}. The methods used here are analogous to similar to those in Section 3 of \cite{EMM}. Our assumption that $(\q, \lin) \in \pairs_n$ will be used in this section as well. Throughout the proof we will assume that $n \ge 4$. We denote by $\RR^n_+$ the set of vectors $v \in \RR^n$ with 
$ \langle v_, e_1 \rangle >0$. The volume of the unit sphere in $\RR^m$ is denoted by $\gamma_{m-1}$. Finally for 
$p+q=n$, we write $c_{p,q}= \frac{2^{(n-2)/2}}{\gamma_{p-1} \gamma_{q-1}}.$


We will start by setting some notation. For $t \in \RR$, recall the one-parameter subgroup of $H$ defined by 
$$a_t=\diag\left(e^{-t}, e^{t}, 1, \ldots, 1\right).$$

Let $f: \RR^{n}_+ \to \RR$ be continuous of compact support. We set  
$$J_f(r, \zeta, s )=\frac 1 {r^{n-3}}\int_{\RR^{n-3}} f(r, x_2, x_3, \ldots, x_{n-1}, s) \ dx_3 \cdots dx_{n-1},
$$
where $x_{2}$ is uniquely determined so that $\q_0(r, x_2, \cdots, x_{n-1}, s) = \zeta$.



\begin{proposition}\label{Jf}
For every $\vep > 0$, there exists $t_0>0$ so that if $t>t_0$,
$$\left| c_{p, q-1} e^{(n-3)t} \int_K f(a_t kv) dk - J_f (\|v\|e^{-t}, \q_0(v), \lin_0(v))\right|< \vep $$
for any $v\in \RR^n$.
\end{proposition}

\begin{proof}
This proposition is analogous to Lemma 3.6 in \cite{EMM}, and a special case of Lemma 5.1 in \cite{Sar}, where the number of linear forms is set to be one, and the matrix $g$ to the identity.  
Let us point out that the function $J_f$ in Lemma 5.1 in \cite{Sar} also depends on the value of quadratic form ($\zeta$ for us), but is not part of the notation.   
\end{proof}

\begin{proposition}\label{sum-bound}
Let $f$ be a continuous bounded function on $\RR^{n}_+$ with compact support. For every $ \epsilon>0$ and $ 
g_0 \in G$, the following inequality holds for sufficiently large values of $t$:
\[ \left|  e^{-(n-3)t} \sum_{ v \in \ZZ^n} J_f (\|g_0 v\|e^{-t}, \q_0(g_0 v), \lin_0(g_0 v))
- c_{p, q-1} \int_K \widetilde{ f}(a_t kg_0) dk \right|< \vep. \]
\end{proposition}

\begin{proof}
It follows from Proposition \ref{Jf} that the number of the terms involved in the sum over vectors in $\ZZ^n$ is $O( e^{(n-3) t})$.  
Now, the desired inequality follows by applying the conclusion of Proposition \ref{Jf} to the vectors $g_0 v $ with $v \in \ZZ^n$ and summing 
over all these vectors. 
\end{proof}

The next proposition is similar to Lemma 3.8 of \cite{EMM}, with the difference that the last variable $s$ is fixed. 

\begin{proposition}\label{EMM Lemma 3.9 (i)} Let $h=h\left(v, \zeta, s \right): \left(\RR^{n} \setminus \{ 0\}\right)\times\RR\times\RR \to \RR$ be a continuous function of compact support.  Then
\begin{equation}\label{integral-formula}
\lim_{T\rightarrow \infty}\frac 1 {T^{n-3}}\int_{\RR^n} h\left(\frac v T, \lin_0(v), \q_0(v)\right)dv = c_{p,q-1} \int_{K}\int_{\RR}\int_{\RR}\int_0^\infty h(rk^{-1}e_1, \zeta, s)r^{n-3}\frac {dr}{2r} ds \, d\zeta \, dm(k).
\end{equation}
\end{proposition}
\begin{proof}
We start from the left-hand side of \eqref{integral-formula}. 
Decompose the vector $v$ as $v=v'+v_ne_n$ and denote by $\q'_0$ the restriction of $\q_0$ to the hyperplane ${v_n=0}$. Since $h$ is compactly supported,
\begin{equation*}
\begin{split}
\lim_{T\to \infty} \int_{v_n}\int_{\RR^{n-1}} h\left(\frac {v'+v_ne_n} T, v_n, \q'_0(v')-v_n^2\right)dv'dv_n&=\lim_{T\to \infty}\int_{v_n}\int_{\RR^{n-1}} h\left(\frac {v'} T, v_n, \q'_0(v')-v_n^2\right)dv'dv_n\\
&=\int_{v_n}\int_{\RR^{n-1}} h_{v_n}\left(\frac {v'} T, \q'_0(v')\right)dv'dv_n,
\end{split}
\end{equation*}
where $h_a(v', \xi)=h(v, a, \xi-a^2)$. Note that $h_a$ is a function on $\left(\RR^{n-1}-\{0\}\right)\times \RR$. By Lemma 3.6 in \cite{EMM}, 
\begin{equation*}
\begin{split}
\int_{v_n}\int_{\RR^{n-1}} h_{v_n}\left(\frac {v'} T, \q'_0(v')\right)dv'dv_n
&=\int_{v_n}\int_K\int_{\RR}\int_0^{\infty} h_{v_n}(rk^{-1}e_1, \xi-v_n^2)r^{n-3}\frac {dr}{2r} d\zeta dm(k)\\
&=\int_{\RR}\int_K\int_{\RR}\int_0^{\infty} h(rk^{-1}e_1, \eta, \xi)r^{n-3}\frac {dr} {2r} d\zeta d\eta,
\end{split}
\end{equation*}
after appropriate changing of variables.
\end{proof}

\begin{corollary}\label{volume}
Let $f$ be a continuous bounded function on $\RR^{n}_+$ with compact support. Set $h(v, \xi, s)= J_f( \| v \|, \xi, s)$. Then 
we have 
\[ \lim_{T\rightarrow \infty}\frac 1 {T^{n-3}}\int_{\RR^n} h  \left( \frac{v}{T},  \q_0(v), \lin_0(v) \right) 
\ dv = c_{p, q-1} \int_{G/\Gamma} \widetilde{ f}(g) d\mu(g).   \]
\end{corollary}

\begin{proof}
Using the change of variable
\[ v=(v_1, \dots, v_n) \mapsto ( v_1, \zeta, v_3, \dots, v_{n}) \] 
where $ \zeta= \q_0( x_1, \dots, x_n)$, the desired 
claim will follow. 


\end{proof}

\begin{corollary}\label{volume2}
Let ${\mathcal{V} }_{T, I, J}(\q, \lin)$ denote the volume of
the subset of $\RR^n$  consisting of  vectors $v$ for which $\|v\| < T$, $ \q(v) \in I $ 
and $ \lin(v) \in J$. Then 
\[ \lim_{T \to \infty}  \frac{{\mathcal{V} }_{T, I, J}(\q, \lin)}{T^{n-3}}
= C(\q, \lin) |I| \ |J|, \]
where $C(\q, \lin)$ is a constant depending only on $\q, \lin$, and $| \cdot |$ denotes the length of an interval. 
\end{corollary}

\begin{proof}
This follows from Corollary \ref{volume}. For details see \cite{BGH}.
\end{proof}

Let us now turn to the proof of the main theorem. Let $g_0 \in G$ be such that $ \q= \q_0^{g_0}$ 
and $ \lin= \lin_0^{g_0}$. 
Consider the space $ \C$ of all functions on $ (\RR^n \setminus \{ 0 \} ) \times \RR \times \RR$ that vanish
outside of a fixed compact set, and equip it with the topology of uniform convergence. It follows from Proposition~\ref{EMM Lemma 3.9 (i)}
that the functional $L: \C \to \RR$ defined by 
\[ L(h)= \lim_{T \to \infty} \frac{1}{T^{n-3}}  \int_{\RR^n} h\left( \frac{v}{T}, \q_0(v), \lin_0(v) \right) \ dv\]
is continuous. Let $\chi$ denote the characteristic function of 
$ \{ v \in \RR^n: \| v \| \in (1/2,1) \} \times [a,b] \times [c,d]$.
Note that 
\[ \sum_{ v \in \ZZ^n } \chi( e^{-t} v, \q_0( g_0 v), \lin_0(g_0 v)  ) \]
counts the number of $v \in \ZZ^n$ satisfying $ e^{t}/2 \le \|  v \| \le e^t$, $ a \le \q_0( g_0 v) \le b$ and $c \le  \lin_0(g_0 v) \le d$.
Given $ \epsilon>0$, there exists $h_+, h_- \in \C$ such that 
\[ h_-( g_0 v, \zeta, s) \le \chi( g_0v, \zeta,s ) \le h_+( g_0 v, \zeta, s), \text{ and}  \quad | L( h_+)- L( h_-) | < \epsilon. \]
One can easily verify that every compactly supported radial function is of the form $J_f(\|v\|, \zeta, s)$ for some compactly suppported function $f$ defined on $\RR^n_+$ with the similar arguments in \cite[page 109]{EMM}.
By Proposition~\ref{sum-bound}, Theorem 2.3, two variations of Siegel's integral formula (Theorem~\ref{Siegel integral formula} and Theorem~\ref{Siegel integral formula:proper F}, depending on the orbit closures) and  Proposition~\ref{EMM Lemma 3.9 (i)},  there exists $t_0$ such that for $t >t_0$ we have 
\begin{equation}
\left|  e^{- (n-3) t} \sum_{ v \in \ZZ^n } h_{\pm}( e^{-t} g_0 v, \q_0( g_0 v) , \lin_0(g_0 v) ) - L(h_{\pm})   \right| < \epsilon;  
\end{equation}
Clearly, for $t$ sufficiently large we have
\begin{equation}  
\left|  e^{- (n-3) t} \int_{\RR^n}  h_{\pm}( e^{-t} g_0 v, \q_0( g_0 v) , \lin_0(g_0 v) ) - L(h_{\pm})   \right|  < \epsilon.
\end{equation}

We note that when we apply Theorem~\ref{Siegel integral formula:proper F}, since we are considering $J_f$ functions for $f$ supported on $\RR^n_+$, we have that $f(xe_n)=0$ for any $x\in \RR$. After applying Theorem 2.3, it follows that
\[\begin{split}
c_{p,q-1}e^{(n-3)t}\int_{F/\Gamma_{F}} \widetilde{f}(g_0g\Gamma)d\mu_F(g)
&=c_{p,q-1}e^{(n-3)t} \int_{\RR^n} f(g_0v) dv + \sum_{m\in \ZZ-\{0\}} f(g_0(mk_0 g_0^{-1}e_n))\\
&=c_{p,q-1}e^{(n-3)t} \int_{\RR^n} f(v) dv + \sum_{m\in \ZZ-\{0\}} f(mk_0e_n)\\
&=c_{p,q-1}e^{(n-3)t} \int_{\RR^n} f(v) dv
\end{split}\]
so that we can apply Proposition~\ref{EMM Lemma 3.9 (i)}.
It follows that for every $ \theta>0$ for $t>t_0$ we have 
\begin{equation}
\begin{split}      
  (1 - \theta) \int_{\RR^n} h_-( e^{-t} v, \q_0( g_0 v), \lin_0(g_0 v) ) \ dv &  \le  
 \sum_{ v \in \ZZ^n} \chi( e^{-t} v, \q_0( g_0 v), \lin_0(g_0 v)  )  \\ & \le  (1 + \theta) \int_{\RR^n} h_+( e^{-t} v, \q_0( g_0 v),  \lin_0(g_0 v) ) \ dv,  
\end{split}
\end{equation}
which implies that for sufficiently large $T\gg 0$,
\[
(1-\theta)\vol\left(\q^{-1}(I) \cap \lin^{-1}(J)\cap B_T\right)
\le \N_{T, I, J}(\q, \lin)
\le (1+\theta)\vol\left(\q^{-1}(I) \cap \lin^{-1}(J)\cap B_T\right).
\]
Now, the theorem follows from Corollary \ref{volume2}.


\section{Counter-examples}\label{counterexamples}\label{sec:counter}
In this section we provide counterexamples showing that Theorem \ref{thm:main} does not generally hold when $(p,q)=(2,2)$ and $(2,3)$. 
The construction is based on the existence of forms of signature $(2,1)$ and $(2,2)$ for which \eqref{qOpp} fails, as proven 
in \cite{EMM}. 

Let us first consider the $(2,2)$-case.
For an irrational positive real number $\beta$, set 
\begin{equation}
\begin{split}      
 \q_{\beta}(x_1, x_2, x_3, x_4) & = (x_1^2+x_2^2)- \beta x_3^2 - ( \beta x_3+ x_4)^2, \\
\lin_{\beta} (x_1, x_2, x_3, x_4) & =  \beta x_3+ x_4.
\end{split}
\end{equation}
We claim that $(\q_\beta, \lin_\beta)$ belongs to $\pairs_4$. It is clear that both forms are irrational. Suppose that $ \lambda_1  \q_{\beta}+ \lambda_2 \lin_{\beta}^2$ 
is a rational quadratic form. By considering the ratios of the coefficients of monomials $ x_4^2$  and $x_3x_4$ and the term
$x_1^2$ we conclude that 
\[ -1+ \frac{ \lambda_2}{ \lambda_1}, -2 \beta \left( 1- \frac{\lambda_2}{ \lambda_1} \right) \]
must be both rational. This implies that $ \beta$ is rational, which is a contradiction. It is also clear that the restriction of 
$\q_{ \beta}$ to the kernel of $\lin_{ \beta}$ is indefinite. This shows that the pair $( \q_{ \beta}, \lin_{ \beta})$ is of type I. 

Now, consider the quadratic form $$ \q'_{\beta}= x_1^2+ x_2^2- \beta^2 x_3^2.$$
Given any $ \varepsilon>0$ and interval $I=(a, b) \subseteq \RR$, Theorem  2.2. of \cite{EMM} provides a dense set of irrational values for $B \subseteq \RR$ such that for every $ \beta \in B$ there exists  $c>0$ and a sequence $T_j \to \infty$ such that 
\[ \N_{\q, I} (T) > c T_j (\log T_j)^{1- \varepsilon} \]
holds for all $j \ge 1$. Choose $ \beta \in (1/2,1)$ and $I= [ \beta^{-1}, 2]$. Then we can find a subset $L_j \subseteq \ZZ^3$ 
of cardinality at least $c T_j (\log T_j)^{1- \varepsilon} $ such that for every $x=(x_1, x_2, x_3) \in L_j$  we have
\[ x_1^2+ x_2^2 - \beta^2 x_3^2 \in [ \beta^{-1}, 2], \quad  \quad  { x_1^2+ x_2^2 +x_2^2  \le  T_j^2 }. \]
For every $(x_1, x_2, x_3) \in L_j$ choose $x_4 \in \ZZ$ such that $ | \beta x_3+ x_4| \le 1$. Note that this also implies that 
for $j$ sufficiently large we have 
\[ | x_4| \le 1+ |\ \beta x_3| \le  1+ \beta|T_j|  \le  |T_j|.  \]
From here we conclude the following inequalities
\begin{equation}
\begin{split}      
\q_{ \beta} (x_1, x_2, x_3, x_4) &= q'_{ \beta} (x_1, x_2, x_3) - ( \beta x_3+ x_4)^2 \in [-1, 2], \\
\lin_{ \beta} (x_1, x_2, x_3, x_4) & = \beta x_3+ x_4 \in [-1, 1]
\end{split}
\end{equation}
Moreover, 
\[ \|  (x_1, x_2, x_3, x_4) \| \le \sqrt{  2}  T_j. \]
Setting $I= [-1,2], J= [-1, 1]$ and adjusting the constant $c$ slightly the claim follows. 

Forms of signature $(2,3)$ can be dealt with in a similar manner by considering the pair
\[
\begin{split}      
 \q_{\beta}(x_1, x_2, x_3, x_4,x_5) & = (x_1^2+x_2^2)- \beta (x_3^2+x_4^2) - ( \beta x_3+ \beta x_4+ x_5)^2, \\
\lin_{\beta} (x_1, x_2, x_3, x_4,x_5) & =  \beta x_3+ \beta x_4+ x_5.
\end{split}
\] 
We omit the details.


\end{document}